\documentclass[11pt]{article}

\usepackage{geometry}
\usepackage{amscd}                		
\usepackage{graphicx}			
\usepackage{amssymb}
\usepackage{amsmath}
\usepackage{amsthm}
\usepackage{amsfonts}
\usepackage{pb-diagram}
\usepackage{array}
\usepackage{multirow}
\usepackage{dcolumn}
\usepackage{enumerate}
\usepackage{comment}
\usepackage{multicol}
\usepackage{mathrsfs}
\usepackage[all]{xy}
\usepackage{bm}
\usepackage{stmaryrd}
\usepackage[abbrev]{amsrefs}

\theoremstyle{plain}
\newtheorem{thm}{Theorem}[subsection]
\newtheorem{lem}[thm]{Lemma}
\newtheorem{claim}{Claim}
\newtheorem{prop}[thm]{Proposition}
\newtheorem{cor}[thm]{Corollary}

\theoremstyle{definition}
\newtheorem{dfn}[thm]{Definition}
\newtheorem{Q}{Objective}
\newtheorem{ex}[thm]{Example}

\newtheorem{rmk}[thm]{Remark}

\newcommand{\Aut}{\mathop{\mathrm{Aut}}\nolimits}

\newcommand{\nrd}{\mathop{\mathrm{nrd}}\nolimits}
\newcommand{\trd}{\mathop{\mathrm{trd}}\nolimits} 
\newcommand{\tr}{\mathop{\mathrm{tr}}\nolimits} 
\newcommand{\id}{\mathop{\mathrm{id}}\nolimits}
\newcommand{\chara}{\mathop{\mathrm{char}}\nolimits}
\newcommand{\rank}{\mathop{\mathrm{rank}}\nolimits}
\newcommand{\kernel}{\mathop{\mathrm{ker}}\nolimits}
\newcommand{\coker}{\mathop{\mathrm{coker}}\nolimits}
\newcommand{\Z}{\mathbb{Z}}
\def\Q{\mathbb{Q}}
\def\P{\mathbb{P}}
\newcommand{\R}{\mathbb{R}}
\newcommand{\C}{\mathbb{C}}

\newcommand{\re}[1]{\operatorname{Re}(#1)}
\newcommand{\im}[1]{\operatorname{Im}(#1)}
\newcommand{\brr}[1]{\bm{(}#1\bm{)}}
\newcommand{\bsr}[1]{\bm{[}#1\bm{)}}
\newcommand{\brs}[1]{\bm{(}#1\bm{]}}
\newcommand{\bss}[1]{\bm{[}#1\bm{]}}
\newcommand{\bs}{\backslash}

\newcommand{\quat}[2]{\left(\dfrac{#1}{#2}\right)}

\allowdisplaybreaks

\numberwithin{equation}{section}

\title{Geodesic continued fraction for Shimura curves and its periodicity: 
the case of $(2,3,7)$-triangle group
}
\author{Hohto Bekki}
\date{
\small\textit{
Department of Mathematics, Faculty of Science and Technology, Keio University, 3-14-1 Hiyoshi, Kohoku-ku, Yokohama, Kanagawa, 223-8522, Japan \\
Email: bekki@math.keio.ac.jp
}}

\begin{document}
\maketitle

\begin{abstract}
In this paper we study the geodesic continued fraction in the case of the Shimura curve coming from the $(2,3,7)$-triangle group. We construct a certain continued fraction expansion of real numbers using the so-called coding of the geodesics on the Shimura curve, and prove the Lagrange type periodicity theorem for the expansion which captures the fundamental relative units of quadratic extensions of $\Q(\cos(2\pi/7))$ with rank one relative unit groups. We also discuss the convergence of these continued fractions.

\vspace{1mm}

\end{abstract}


\section{Introduction}\label{intro}

Let $\eta:= 2 \cos \left(\frac{2\pi}{7}\right) \in \R$ be the unique positive root of $X^3+X^2-2X-1$. 
The aim of this paper is to present a geometric construction of the continued fraction expansion such as
\begin{align}
 &(1-\eta ^2)\sqrt{\eta} +\sqrt{1+3 \eta -2\eta^2} \nonumber  \\
 &= (1-\eta ^2)\sqrt{\eta} + \eta^2-\eta + \cfrac{2(1+2 \eta -2\eta^2)}
 {\eta^2-\eta+ \cfrac{1+2 \eta -2\eta^2}
 {\eta^2-\eta+\cfrac{1+2 \eta -2\eta^2}
 {\eta^2-\eta+\cfrac{1+2 \eta -2\eta^2}{ \cdots {\text (periodic)}}}}} \label{eqn ex1}
\end{align}
with the Lagrange type periodicity property (Theorem \ref{lagrange closed geod} and Theorem \ref{thm lagrange}). 
Here the term $(1-\eta ^2)\sqrt{\eta}$ on the both sides should not be deleted in order for this continued fraction expansion to have the natural geometric meaning. See Example \ref{ex1}. 

The classical Lagrange theorem in the theory of continued fraction says that the continued fraction expansion of a given real number becomes periodic if and only if the number is a real quadratic irrational, and that the period of the continued fraction expansion describes the fundamental unit of the associated order in the real quadratic field. 
Analogously, we know that a given geodesic on the upper-half plane $\mathfrak h$ becomes a closed geodesic (``periodic'') on the modular curve $SL_2(\Z)\bs \mathfrak h$ if and only if the two end points of the geodesic are conjugate real quadratic irrationals, and that the length of the closed geodesic becomes the regulator of the associated order in the real quadratic field. Cf. Theorem \ref{lagrange1}.

Our motivation for this study is to extend the Lagrange theorem to number fields other than real quadratic fields based on this geometric analogue of the Lagrange theorem. 
In a previous paper \cite{bekki17}, based on this analogy and inspired by the works of Artin~\cite{artin24}, Sarnak~\cite{sarnak82}, Series~\cite{series85}, Katok~\cite{katok96}, Lagarias~\cite{lagarias94}, Beukers~\cite{beukers14}, etc., we have studied the geodesic multi-dimensional continued fraction and its periodicity using the geodesics on the locally symmetric space of $SL_n$. 
As a result we have established a geodesic multi-dimensional continued fraction and a $p$-adelic continued fraction with the Lagrange type periodicity theorems in the case of extensions $E/F$ of number fields with rank one relative unit group, and in the case of imaginary quadratic fields with rank one $p$-unit group, respectively. 
Recently the author found that Vulakh~\cite{vulakh02}, \cite{vulakh04} had also used the same idea to compute the fundamental units of some families of number fields.

In this paper we mainly study the case of the Shimura curve coming from the $(2,3,7)$-triangle group $\Delta(2,3,7)$. 
We construct a continued fraction expansion using the geodesics on the Shimura curve $\Delta(2,3,7)\bs \mathfrak h$, and prove the Lagrange type periodicity theorem for quadratic extensions $K/\Q(\cos(2\pi/7))$ with rank one relative unit groups. 
We refer to such an extension the ``relative rank one'' extension. 
Although these number fields can already be treated in the previous paper \cite{bekki17}, a major difference is that we can actually expand numbers in the form of ``continued fraction'' as in (\ref{eqn ex1}), while the geodesic multi-dimensional continued fraction in \cite{bekki17} only gives a sequence of matrices in $SL_n(\mathbb Z)$. 

The idea of considering the geodesic continued fraction for some arithmetic Fuchsian groups is briefly discussed in \cite{katok86} by Katok. In the final sections (Remark and Examples) of \cite{katok86}, Katok considers the arithmetic groups coming from quaternion algebras over $\Q$, and gives some examples of the periodic geodesic continued fraction expansions (Katok calls this the ``code'') of closed geodesics. 
Some of the essential ideas in this paper are generalization of Katok's ideas to our case of quaternion algebras over totally real fields $F$, especially $\Q(\cos(2\pi/7))$.

Our strategy is as follows. First we extend the geometric analogue of the Lagrange theorem (Theorem \ref{lagrange1}) to Shimura curves (Proposition \ref{geodesic lagrange}). We treat a general Shimura curve, not only the one coming from the $(2,3,7)$-triangle group, since the argument does not change too much. 
Then we restrict ourselves to the Shimura curve $\Delta(2,3,7)\bs \mathfrak h$ and construct the continued fraction expansion which can detect the relative units of the relative rank one quadratic extensions of $\Q(\cos(2\pi/7))$ as the periods of the continued fraction expansion. 

For this purpose, we use the technique called the geodesic continued fraction studied by Series~\cite{series85}, Katok~\cite{katok96}, etc. in the field of reduction theory, dynamical systems, etc. (Some authors refer to the geodesic continued fraction also as the cutting sequence or the Morse coding.)
In order to obtain a convergent continued fraction expansion which is similar to the classical one such as (\ref{eqn ex1}), we slightly modify the original algorithm by considering the regular geodesic heptagon which is the union of some copies of the fundamental domain. 
Here we use the generator of $\Delta(2,3,7)$ given by Elkies~\cite{elkies98}, Katz-Schaps-Vishne~\cite{katzschapsvishne11}.  See Figure \ref{FD(2,3,7)} and Definition \ref{GCF algo}. 
Then we consider the formal continued fraction expansion (\ref{CF}) associated to the geodesic continued fraction, and discuss its convergence. In fact, although the traditional $k$-th convergent of (\ref{CF}) does not converge in general, we show that there is a natural regularization of the $k$-th convergent and prove its convergence. See Theorem \ref{thm convergence}, Corollary \ref{cor convergence} and Corollary \ref{conv cor2}.  

Finally we study the Lagrange type periodicity of our continued fraction expansion. 
We prove two versions of the periodicity theorem: the first version Theorem \ref{lagrange closed geod} is about the closed geodesics on the Shimura curve $\Delta(2,3,7)\bs \mathfrak h$, and the second refined version Theorem \ref{thm lagrange} is about geodesics not necessarily closed on $\Delta(2,3,7)\bs \mathfrak h$. 
Note that it is a well known fact that the geodesic continued fraction (or the cutting sequence/Morse code) of a ``generic'' geodesic becomes periodic if and only if the geodesic becomes closed geodesic on the quotient space. Cf. Katok-Ugarcovici~\cite{katokugarcovici07}*{p.94}. 
Therefore, (1) of Theorem \ref{lagrange closed geod}, i.e., the periodicity of the geodesic continued fraction expansion itself, follows naturally from Proposition \ref{geodesic lagrange} and this fact. 
On the other hand, we need more argument for the latter part of Theorem \ref{lagrange closed geod} about the fundamental unit. 
Actually, we have to take care about the vertices of the fundamental domain in order to obtain the fundamental units as a minimal period of the continued fraction expansions.   
We also need some delicate arguments for the periodicity in Theorem \ref{thm lagrange}.  
See also Remark \ref{rmk beta free}.

\paragraph{Remark on some relevant preceding studies}

There are many literatures which study the closed geodesics on hyperbolic surfaces (not necessarily Shimura curves) using the periodic geodesic continued fractions or the cutting sequences/Morse codes. 

For example, in \cite{series86}, \cite{katok86}, \cite{abramskatok19}, Series, Katok and Abrams-Katok study the reduction theory for the general Fuchsian groups or the symbolic dynamics associated to the geodesic flows on hyperbolic sufaces using the cutting sequence for (certain classes of) Fuchsian groups.  
As we have mentioned above, in \cite{katok86}, Katok considers the case where the Fuchsian group is coming from the quaternion algebras over $\Q$. 
Some new features in this paper are to give the continued fraction expression such as (\ref{CF}) and to establish an explicit correspondence between the period of continued fraction expansions and the fundamental relative unit of certain quadratic extensions over the totally real field $\Q(\cos(2\pi/7))$. 

As another example, in \cite{vogeler03}, Vogeler studies the closed geodesics on the Hurwitz surface ($= \Delta(2,3,7)\bs \mathfrak h$). He associates each ``edge path'' the hyperbolic element in the $(2,3,7)$-triangle group $\Delta(2,3,7)$, and hence the geodesic on $\mathfrak h$ which becomes closed on $\Delta(2,3,7)\bs \mathfrak h$. On the other hand the ``edge path'' admits a very simple combinatorial description using the words consisting of $R$ and $L$. Using this correspondence, he combinatorially studies the length spectra of closed geodesics on $\Delta(2,3,7)\bs \mathfrak h$. 
In this paper, we basically study the reverse direction, that is, we input the geodesics (not necessarily closed) and output the $RL$-sequences or the hyperbolic elements if the geodesic continued fraction becomes periodic, and discuss its relation to unit groups of the quadratic extension of $\Q(\cos(2\pi/7))$.

%

To sum up, the upshots of this paper are 
\begin{enumerate}[(1)]
\item to give a new explicit presentation of any real number as a ``convergent'' continued fraction such as (\ref{eqn ex1}) or (\ref{CF}) which becomes periodic if and only if the number is a certain algebraic number, and 
\item to establish the correspondence between the periods of such continued fractions and the fundamental relative units of the ``relative rank one'' extensions of $\Q(\cos(2\pi/7))$, 
\end{enumerate}
by using the arithmetic and geometric properties of the Shimura curve $\Delta(2,3,7)\bs \mathfrak h$. 

\section{Preliminaries on Shimura curves}\label{section shimura curve}

Let $F$ be a totally real number field of degree $d \geq 1$ and let $\mathcal O_F$ be the ring of integers of $F$. We denote by $\sigma_1,\dots, \sigma_d$ the set of archimedean places of $F$. We also denote by  $\sigma_i: F\hookrightarrow F_{\sigma_i}:=\R$ $(1\leq i\leq d)$ the completion map of $F$ at $\sigma_i$. Let $A$ be a quaternion algebra over $F$ and let $\mathcal O \subset A$ be a maximal order, i.e., an $\mathcal O_F$-subalgebra of $A$ which is finitely generated as an $\mathcal O_F$-module such that $\mathcal O \otimes_{\mathcal O_F} F=A$, and not properly contained in any other such $\mathcal O_F$-subalgebra. 
We denote by $\mathcal O^1:=\{x\in \mathcal O^{\times} \mid \nrd (x)=1\}$ the group of reduced norm one units in $\mathcal O$, where $\nrd: A\rightarrow F$ is the reduced norm on $A$. 
 Suppose that $A$ is unramified at $\sigma_1$ and ramified at $\sigma_2, \dots, \sigma_d$, i.e., $A\otimes_F F_{\sigma_1}\simeq M_2(F_{\sigma_1})$ and $A\otimes_F F_{\sigma_i}\simeq \mathbb H$ (the Hamilton quaternion) for $i=2, \dots, d$. Let us fix such an isomorphism (as $F_{\sigma_1}$-algebras)
\begin{align}\label{iota}
\iota :A\otimes_F F_{\sigma_1}\overset{\sim}{\rightarrow} M_2(F_{\sigma_1})=M_2(\R),
\end{align}
and set $\Gamma_{\mathcal O} := \iota(\mathcal O^1)$. 
By the definition of the reduced norm, $\Gamma_{\mathcal O}$ is a subgroup of $SL_2(\R)$ and acts on the upper-half plane by the linear fractional transformation.

To be precise, we denote by $\mathfrak h:=\{z=x+\sqrt{-1}y \in \C \mid \im z=y>0\} \subset \C$ the upper-half plane. We naturally embed $\mathfrak h$ into $\P^1(\C) :=\C \cup \{\infty\}$, the complex projective line with the usual topology as a manifold. Then the boundary $\partial \mathfrak h$ of $\mathfrak h$ in $\P^1(\C)$ becomes $\P^1(\R):=\R \cup \{\infty\}$, and we denote by $\overline{\mathfrak h}:= \mathfrak h \cup \P^1(\R)$ the compactified upper-half plane in $\P^1(\C)$. 
The group $GL_2(\R)$ acts on $\P^1(\C)$ by the linear fractional transformation:
\begin{align}
\gamma z =\frac{az+b}{cz+d}  \quad \text{ for} \quad 
\gamma=
\begin{pmatrix}
a & b \\
c & d \\
\end{pmatrix}
\in GL_2(\R), z \in \P^1(\C),
\end{align}
and the action of $SL_2(\R)$ preserves $\mathfrak h$ and $\P^1(\R)$. Thus $SL_2(\R)$ also acts on $\mathfrak h$ and $\P^1(\R)$ by the linear fractional transformation. 
We also equip $\mathfrak h$ with the Poincar\'e metric $ds^2=\frac{dx^2+dy^2}{y^2}$ ($z=x+\sqrt{-1}y \in \mathfrak h$). The action of $SL_2(\R)$ on $\mathfrak h$ is preserves this metric, and hence preserves the geodesics on $\mathfrak h$.

Then it is known that $\Gamma_{\mathcal O}$ acts properly discontinuously on $\mathfrak h$, and the quotient space $\Gamma_{\mathcal O}\bs \mathfrak h$ has a canonical structure of an algebraic curve over $\overline{\Q}$. See Shimura~\cite{shimura67}. Algebraic curves obtained in this way are called the \textit{Shimura curves} (of level 1).

In the following, for $\alpha, \beta \in \P^1(\R)=\R\cup \{\infty\}$ such that $\alpha\neq \beta$, we mean by {\it the oriented geodesic on $\mathfrak h$ joining $\beta$ to $\alpha$} the geodesic on $\mathfrak h$ joining $\alpha$ and $\beta$ equipped with the orientation from $\beta$ to $\alpha$, and denote by $\varpi_{\beta \to \alpha}$.

\subsection{The modular curve}\label{sect modular curve}
In this subsection we recall the case of the modular curve $SL_2(\Z)\bs \mathfrak h$ as an example of Shimura curve and  explain the geometric interpretation of the Lagrange theorem which is the key idea for our generalization of the Lagrange theorem. 

We consider the case where $F=\Q, A=M_2(\Q)$ and $\mathcal O= M_2(\Z) \subset A$. We choose the canonical base change isomorphism $\iota=\id : M_2(\Q)\otimes \R \simeq M_2(\R)$ as an identification (\ref{iota}). Then we have $\Gamma_{\mathcal O}= SL_2(\Z)$ and the resulting Shimura curve $\Gamma_{\mathcal O}\bs \mathfrak h=SL_2(\Z) \bs \mathfrak h$ is the classical modular curve. 
Then the following is known about the closed geodesics on the modular curve $SL_2(\Z) \bs \mathfrak h$. See Sarnak~\cite{sarnak82} for example.

\begin{thm}[The geodesic Lagrange theorem]\label{lagrange1}~
\begin{enumerate}[$(1)$]
\item Let $\alpha, \beta \in \R \cup \{\infty\} =\partial \mathfrak h$ such that $\alpha \neq \beta$, and let $\varpi$ be the oriented geodesic on the upper-half plane $\mathfrak h$ joining $\beta$ to $\alpha$. We denote by $\overline{\varpi}$ the projection of the geodesic $\varpi$ on the modular curve. The following conditions are equivalent.
\begin{enumerate}[{\rm (i)}]
\item The projected geodesic $\overline{\varpi}$ becomes a closed geodesic, i.e., $\overline{\varpi}$ has a compact image in $SL_2(\Z)\bs \mathfrak h$.
\item There exists a hyperbolic element $\gamma \in SL_2(\Z)$ (i.e., $\gamma$ has two distinct real eigenvalues) such that $\gamma \varpi=\varpi$, i.e., $\gamma \alpha=\alpha$ and $\gamma \beta =\beta$.
\item The end points $\alpha,\beta$ are real quadratic irrationals conjugate to each other over $\Q$. 
\end{enumerate}
\item Suppose that the above conditions are satisfied. 
Let $\Gamma_{\varpi}:=\{\gamma \in SL_2(\Z) \mid \gamma \varpi=\varpi\}$ be the stabilizer subgroup of $\varpi$ in $SL_2(\Z)$, and define an order $\mathcal O_{\alpha}$ in the real quadratic field $\Q(\alpha)$ by
$\mathcal O_{\alpha}:=\{x \in \Q(\alpha) \mid x(\Z\alpha+\Z) \subset (\Z\alpha+\Z)\}$.
We denote by $\mathcal O_{\alpha}^1:=\{x \in \mathcal O_{\alpha}^{\times}\mid N_{\Q(\alpha)/\Q}(x)=1\}$ the group of norm one units. Then the following natural map is an isomorphism of groups.
\begin{align}
\Gamma_{\varpi} \overset{\sim}{\rightarrow} \mathcal O_{\alpha}^1; 
\begin{pmatrix}
a&b\\
c&d\\
\end{pmatrix}
\mapsto c\alpha+d.
\end{align}
\end{enumerate}
\end{thm}

Recall that the classical Lagrange theorem says that the continued fraction expansion of a real number $\alpha$ becomes periodic if and only if $\alpha$ is a real quadratic irrational, and we can actually compute the fundamental unit of $\mathcal O_{\alpha}$ from the period of continued fraction expansion of $\alpha$. Therefore, Theorem \ref{lagrange1} can be seen as a geometric interpretation of the Lagrange theorem.

In the following we first extend Theorem \ref{lagrange1} to the Shimura curves $\Gamma_{\mathcal O}\bs \mathfrak h$ explicitly (Proposition \ref{geodesic lagrange}). Then we restrict ourselves to the special case where $\Gamma_{\mathcal O}$ becomes the so called $(2,3,7)$-triangle group, and construct our continued fraction explicitly. We first discuss the convergence of the continued fraction expansion, and then deduce the Lagrange type periodicity theorem from Proposition \ref{geodesic lagrange}.

\subsection{Closed geodesics on Shimura curves}\label{subsection closed geodesics}
Now we return to the general setting and use the notations in the beginning of Section \ref{section shimura curve}, i.e., $F$ is a totally real field of degree $d$, $A$ is a quaternion algebra over $F$ such that $A\otimes_{\Q}\R\simeq M_2(\R) \times \mathbb H^{d-1}$, $\mathcal O\subset A$ is a maximal order of $A$, etc. In the following we also assume that $A \not\simeq M_2(F)$ (hence $A$ is a division algebra), since the case where $A\simeq M_2(F)$ has already explained in the previous subsection. (Note that if $A\simeq M_2(F)$ then $F$ must be $\Q$ by the assumption $A\otimes_{\Q}\R\simeq M_2(\R) \times \mathbb H^{d-1}$.)

For simplicity, we regard $F$ as a subfield of $\R$ via the embedding $\sigma_1:F \hookrightarrow F_{\sigma_1}=\R$.
In order to extend Theorem \ref{lagrange1} we fix the identification $\iota: A\otimes_{F}\R \overset{\sim}{\rightarrow} M_2(\R)$ (\ref{iota}) explicitly as follows. Since $\chara F=0\neq 2$, the quaternion algebra $A$ is isomorphic to $\quat{a,b}{F}$ for some $a,b \in F^{\times}$. Here $\quat{a,b}{F}$ is the quaternion algebra generated by the basis $1, i, j, k$ of the following form.  
\begin{align}
&\quat{a,b}{F}=F+Fi+Fj+Fk \\
&~~ i^2=a, j^2=b, ij=-ji=k.
\end{align}
By the assumptions $A\not\simeq M_2(F)$ and $A\otimes_F \R\simeq M_2(\R)$, we have $a,b \notin (F^{\times})^2$ and we may assume $a,b>0$. We take a splitting field $L:=F(\sqrt{b})\subset \R$. For $z \in L$ we denote by $\bar{z}$ the conjugate of $z$ over $F$, i.e.,
\begin{align}
\bar{~}:L\rightarrow L; z=x+y\sqrt{b} \mapsto \bar{z} = x-y\sqrt{b} \quad (x,y \in F).
\end{align}
Then we have an embedding
\begin{align}
\iota: A\simeq \quat{a,b}{F} \hookrightarrow M_2(L)\subset M_2(\R); 
\begin{cases}
1 &\mapsto 
\begin{pmatrix}
1&0\\
0&1\\
\end{pmatrix}\vspace{1mm}\\
i &\mapsto 
\begin{pmatrix}
0&a\\
1&0\\
\end{pmatrix}\vspace{1mm}\\
j &\mapsto 
\begin{pmatrix}
\sqrt{b}&0\\
0&-\sqrt{b}\\
\end{pmatrix}
\end{cases}\label{iota2}
\end{align}
as $F$-algebras which induces an isomorphism $\iota: A\otimes_F \R \overset{\sim}{\rightarrow}M_2(\R)$. Note that the image of $A$ under $\iota$ can be described as follows:
\begin{align}\label{image of A}
\iota: A \overset{\sim}{\rightarrow} \left\{ \left.
\begin{pmatrix}
z&a\bar{w}\\
w&\bar{z}
\end{pmatrix}\right|
z,w \in L
\right\} \subset M_2(L).
\end{align}
In the following we regard $A$ as a subalgebra of $M_2(L)$ via this identification. 
Then the reduced norm and the reduced trace on $A$ is nothing but the restriction of the determinant and the trace on $M_2(L)$ respectively, i.e., 
\begin{align}
\nrd=\det: &A \rightarrow F; 
\begin{pmatrix}
z&a\bar{w}\\
w&\bar{z}
\end{pmatrix}
\mapsto z\bar z -aw\bar w, \\
\trd=\tr: &A \rightarrow F; 
\begin{pmatrix}
z&a\bar{w}\\
w&\bar{z}
\end{pmatrix}
\mapsto z + \bar z.
\end{align}

Now let $\alpha, \beta \in \R\cup \{\infty\} =\partial \mathfrak h$ such that $\alpha \neq \beta$, and let $\varpi_{\beta \to \alpha}$ be the oriented geodesic on $\mathfrak h$ joining $\beta$ to $\alpha$. We denote by $\overline{\varpi}_{\beta \to \alpha}$ the projection of $\varpi_{\beta \to \alpha}$ on the Shimura curve $\Gamma_{\mathcal O}\bs \mathfrak h$. Let 
\begin{align}
\Gamma_{\varpi_{\beta \to \alpha}}:=\{\gamma\in \Gamma_{\mathcal O}\mid \gamma \varpi_{\beta\to\alpha}= \varpi_{\beta\to\alpha} \text{, i.e., } \gamma \alpha=\alpha \text{ and } \gamma \beta =\beta\}
\end{align}
be the stabilizer subgroup of $ \varpi_{\beta\to\alpha}$ in $\Gamma_{\mathcal O}$. 
We recall the following elementary fact. 

\begin{lem}\label{lem hyp}
An element $\gamma \in \Gamma_{\varpi_{\beta \to \alpha}}$ is a hyperbolic element (i.e., an element with distinct real eigenvalues) if and only if $\gamma \neq \pm 1$.
\end{lem}
\begin{proof}
Let $\gamma \in \Gamma_{\varpi_{\beta \to \alpha}}$. First note that $\gamma$ is diagonalizable in $M_2(\R)$ since it has two distinct fixed points $\alpha, \beta \in \P^1(\R)$. More precisely, let 
\begin{align}
v=\begin{pmatrix}
\alpha_1\\
\alpha_2
\end{pmatrix}, \ 
w=\begin{pmatrix}
\beta_1\\
\beta_2
\end{pmatrix} \in \R^2\!-\!\{0\}, 
\end{align}
be the eigenvectors of $\gamma$ corresponding to the distinct fixed points $\alpha$, $\beta$ respectively, i.e., $\alpha=[\alpha_1: \alpha_2], \beta=[\beta_1:\beta_2]$ in $\P^1(\R)$. Let $\lambda, \mu \in \R$ be the eigenvalues of $\gamma$ corresponding to $v, w$ respectively. Then we have
\begin{align}
\gamma = 
\begin{pmatrix}
\alpha_1&\beta_1\\
\alpha_2&\beta_2
\end{pmatrix}
\begin{pmatrix}
\lambda&0\\
0&\mu
\end{pmatrix}
\begin{pmatrix}
\alpha_1&\beta_1\\
\alpha_2&\beta_2
\end{pmatrix}^{-1}. 
\end{align}
Now, if $\lambda=\mu$, then $\lambda=\mu=\pm1$ since $\det \gamma =1$, and hence $\gamma=\pm 1$. Therefore, we see that $\gamma$ has distinct real eigenvalues if and only if $\gamma \neq \pm 1$. 
\end{proof}
%
The following proposition extends Theorem \ref{lagrange1} to the Shimura curve $\Gamma_{\mathcal O}\bs \mathfrak h$.

\begin{prop}[The geodesic Lagrange theorem for Shimura curves]\label{geodesic lagrange}
Let the notations be as above. Then the following conditions are equivalent.
\begin{enumerate}[{\rm (i)}]
\item The projection $\overline{\varpi}_{\beta \to \alpha}$ becomes a closed geodesic, i.e., $\overline{\varpi}_{\beta \to \alpha}$ has a compact image in $\Gamma_{\mathcal O}\bs \mathfrak h$.
\item There exists a hyperbolic element in $\Gamma_{ \varpi_{\beta\to\alpha}}$, i.e., $\Gamma_{ \varpi_{\beta\to\alpha}} \neq \{\pm 1\}$.
\item The two endpoints $\alpha$ and $\beta$ are of the following form:
\begin{align}\label{alpha beta}
\begin{cases}
\alpha &=\frac{1}{2w}(z-\bar z \pm \sqrt{D_{z,w}}) \\
\beta &= \frac{1}{2w}(z-\bar z \mp \sqrt{D_{z,w}})
\end{cases}
\end{align}
for some $z,w\in L$ such that $D_{z,w}:=(z-\bar z)^2+4aw\bar w > 0$. 
Here if $w=0$, we assume $(\alpha, \beta)=(0,\infty)$ or $(\infty,0)$.
\end{enumerate}
\end{prop}
%


Before proving this proposition we introduce some more notations. 
Suppose that $\alpha, \beta \in \R \cup\{\infty\}$ can be written in the form 
\begin{align}
\alpha &= \frac{1}{2w}(z-\bar z \pm \sqrt{D_{z,w}}) \\
\beta &= \frac{1}{2w}(z-\bar z \mp \sqrt{D_{z,w}})
\end{align}
for $z,w \in L$ such that $D_{z,w}=(z-\bar z)^2+4aw\bar w >0$. Note that we have $D_{z,w} \in F$ by the definition. 
Set $\theta_{z,w}:=
\begin{pmatrix}
z&a\bar w\\
w&\bar z\\
\end{pmatrix} \in A \subset M_2(L)$. 
Then we define 
\begin{align}
K_{z,w}:=F[\theta_{z,w}] \subset A
\end{align}
to be the $F$-subalgebra of $A$ generated by $\theta_{z,w}$, and set $\mathcal O_{z,w}:=K_{z,w}\cap \mathcal O$. Note that we have $\theta_{z,w} \notin F$ because $D_{z,w} \neq 0$.
We denote by $\mathcal O_{z,w}^1:=\mathcal O_{z,w}^{\times} \cap \mathcal O^1$ the group of reduced norm one units in $\mathcal O_{z,w}$.

\begin{lem}\label{lem ass field}
\begin{enumerate}[$(1)$]
\item The subalgebra $K_{z,w}$ is a maximal (commutative) subfield in $A$. 
\item The field $K_{z,w}$ is a quadratic extension of $F$, and the reduced norm on $A$ restricted to $K_{z,w}$ coincides with the field norm of $K_{z,w}/F$.
\item The field $K_{z,w}$ splits at the place $\sigma_1$ and ramifies at the places $\sigma_2, \dots, \sigma_d$, i.e., $K_{z,w} \otimes_F \R \simeq \R \times \R$ and $K_{z,w} \otimes_F F_{\sigma_i}\simeq \C$ for $2\leq i \leq d$. 
\item Let $F(\sqrt{D_{z,w}}) \subset \R$ be the quadratic extension of $F$ in $\R$ generated by $\sqrt{D_{z,w}}$. Then we have the following isomorphism of fields:
\begin{align}
\rho_{\alpha}: K_{z,w}\overset{\sim}{\rightarrow} F(\sqrt{D_{z,w}}); 
\begin{pmatrix}
q&r\\
s&t\\
\end{pmatrix} \mapsto s\alpha +t. \label{ev map} 
\end{align}
\item We have the following identity:
\begin{align}
K_{z,w}^{\times} 
&=\{ \gamma \in A^{\times}\subset GL_2(L) \mid \gamma \alpha=\alpha, \gamma \beta=\beta \} \label{eqn217} \\
&=\{ \gamma \in A^{\times}\subset GL_2(L) \mid \gamma \alpha=\alpha \} \label{eqn218}
\end{align}
In particular, the subfield $K_{z,w} \subset A$ depends only on $\alpha $, and does not depend on the choice of $z,w\in L$. By taking the intersection with $\mathcal O^1$ we also obtain $\mathcal O_{z,w}^1=\Gamma_{\varpi_{\beta \to \alpha}}$.
\item  The subring $\mathcal O_{z,w} \subset K_{z,w}$ is an order in $K_{z,w}$. In particular $\rank_{\Z} \mathcal O_{z,w}^1=1$, and there exists $\varepsilon_0 \in \mathcal O_{z,w}^1$ such that $\mathcal O_{z,w}=\{\pm \varepsilon_0^k \mid k \in \Z\}$.
\end{enumerate}
\end{lem}
\begin{proof}
(1) This is because we have assumed that $A$ is a division algebra and $\theta_{z,w} \notin F$. 

(2) Now since $\theta_{z,w} \notin F$, the characteristic polynomial $P_{z,w}(X)=X^2-\tr(\theta_{z,w})X+\det(\theta_{z,w})=X^2-\trd(\theta_{z,w})X+\nrd(\theta_{z,w}) \in F[X]$ of $\theta_{z,w}$ as a matrix in $M_2(L)$ becomes the minimal polynomial of $\theta_{z,w}$ with respect to the field extension $K_{z,w}/F$. Therefore, $K_{z,w}$ is a quadratic extension of $F$, and the reduced norm and the field norm coincide. 

(3) We easily see that the discriminant of the characteristic polynomial $P_{z,w}(X)$ is $\tr(\theta_{z,w})^2-4\det(\theta_{z,w})= D_{z,w}$. Therefore the assumption $D_{z,w}>0$ implies that $K_{z,w}/F$ splits at $\sigma_1$. On the other hand, for $2\leq i \leq d$, the assumption $A\otimes_F F_{\sigma_i}=\mathbb H$ implies that $K_{z,w} \otimes_F F_{\sigma_i}$ must be a field of degree $2$ over $\R$, and hence isomorphic to $\C$. 

(4) The map $\rho_{\alpha}$ is an $F$-linear map which sends $1$ to $1$, and $\theta_{z,w}$ to $w\alpha+\bar z$. Now, since $w\alpha+\bar z=\frac{1}{2}(z+\bar z \pm \sqrt{D_{z,w}})$ is a root of the characteristic polynomial $P_{z,w}(X)$, the map $\rho_{\alpha}$ is an isomorphism.

(5) 
Note that the fixed points of $\theta_{z,w}$ in $\P^1(\C)$ are $\alpha$ and $\beta$, i.e., $\theta_{z,w}\alpha=\alpha$ and $\theta_{z,w}\beta=\beta$. 
Let $\gamma \in K_{z,w}^{\times}$. Then $\gamma$ commutes with $\theta_{z,w}$ in $K_{z,w} \subset M_2(L)$, and hence $\gamma$ and $\theta_{z,w}$ have the same eigenvectors. Therefore the fixed points of $\gamma$ are also $\alpha$ and $\beta$, and thus $\gamma$ belongs to the right hand side of (\ref{eqn217}). 
Clearly, the right hand side of (\ref{eqn217}) is a subset of the right hand side of (\ref{eqn218}). 
%
%
Now, let $\gamma \in A^{\times}$ such that $\gamma \alpha=\alpha$. It suffices to show that $\gamma \in K_{z,w}=F[\theta_{z,w}]$. Suppose $\gamma \notin F[\theta_{z,w}]$. Then, since $F[\theta_{z,w}]$ is a field by (1), the $F$-subalgebra $F[\theta_{z,w},\gamma] \subset A$ becomes an $F[\theta_{z,w}]$-algebra of degree at least $2$. Therefore we obtain $F[\theta_{z,w},\gamma]=A$ by (2). By the assumption, $\theta_{z,w}$ and $\gamma$ share the same fixed point $\alpha$, and hence share the same eigenvector, say $v \in \R^2\!-\!\{0\}$. Then it follows that every element of $A^{\times} \subset GL_2(\R)$ shares the same eigenvector $v$, and hence every element of $(A\otimes_F \R)^{\times}=GL_2(\R)$ shares the same eigenvector $v$. However, this is impossible. Thus we see $\gamma \in K_{z,w}$.

(6) Since $\mathcal O$ is a finitely generated $\mathcal O_F$-module and $\mathcal O_F$ is noetherian, we see that $\mathcal O_{z,w}$ is finitely generated as an $\mathcal O_F$-module. On the other hand, since $\mathcal O$ is an order in $A$, there exists $m \in \Z_{>0}$ such that $m\theta_{z,w} \in \mathcal O\cap K_{z,w}=\mathcal O_{z,w}$, thus we see $\mathcal O_{z,w}\otimes_{\mathcal O_F} F=K_{z,w}$. Therefore $\mathcal O_{z,w}$ is an order in $K_{z,w}$.
Now, by (2), we have $\mathcal O_{z,w}^1=\kernel (N_{K_{z,w}/F}: \mathcal O_{z,w}^{\times}\rightarrow \mathcal O_F^{\times})$, and $\coker(N_{K_{z,w}/F})$ is a torsion group. Thus we get $\rank_{\Z} \mathcal O_{z,w}^1=1$ by (3) and Dirichlet's unit theorem.
\end{proof}

We denote by $R_{\mathcal O,z,w}$ (resp. $U_{\mathcal O,z,w /F}$) the image of $\mathcal O_{z,w}$ (resp. $\mathcal O_{z,w}^1$) under the isomorphism $\rho_{\alpha}$, i.e., 
\begin{align}
R_{\mathcal O, z,w} &:= \rho_{\alpha}(\mathcal O_{z,w}) \subset F(\sqrt{D_{z,w}}), \\
U_{\mathcal O, z,w/F} &:= \rho_{\alpha}(\mathcal O_{z,w}^1)=\ker (N_{F(\sqrt{D_{z,w}})/F} : R_{\mathcal O, z,w}^{\times} \rightarrow \mathcal O_F^{\times}).
\end{align}

\begin{proof}[Proof of Proposition \ref{geodesic lagrange}]
The equivalence (i) $\Leftrightarrow$ (ii) is clear. The implication (ii) $\Rightarrow$ (iii) is also clear because if $\gamma \in \Gamma_{\varpi_{\beta\to\alpha}}$ is hyperbolic, $\gamma$ can be written as 
$\gamma =\theta_{z,w}= \begin{pmatrix}
z&a\bar w\\
w&\bar z\\
\end{pmatrix}$  for  some $z,w \in L$ with $D_{z,w}>0$ by (\ref{image of A}). Then we easily see that the two fixed points $\alpha, \beta$ of $\gamma$ can be written as (\ref{alpha beta}).
It remains to prove (iii) $\Rightarrow$ (ii). 
Suppose $\alpha$ and $\beta$ are written as (\ref{alpha beta}). By Lemma \ref{lem ass field} (6), there exists a non-torsion unit $\varepsilon \in \mathcal O_{z,w}^1$. Then, by Lemma \ref{lem ass field} (5), we see $\varepsilon \in \Gamma_{\varpi_{\beta\to \alpha}}$. 
Finally, because $\varepsilon \neq \pm 1$, it is a hyperbolic element by Lemma \ref{lem hyp}.
\end{proof}

\subsection{The Shimura curve coming from the $(2,3,7)$-triangle group}\label{ex (2,3,7)}

Here we recall some basic facts about the case where $\Gamma_{\mathcal O}$ becomes the $(2,3,7)$-triangle group. 
Let $\eta :=2\cos\left(\frac{2\pi}{7}\right) \in \R$ be the unique positive root of $X^3+X^2-2X-1$, and let $F:=\Q (\eta)\subset \R$ be the totally real cubic field generated by $\eta$ over $\Q$. Then we have $\mathcal O_F=\Z[\eta]$. 
We consider the quaternion algebra $A:=\quat{\eta,\eta}{F}=F+Fi+Fj+Fk$ with $i^2=j^2=\eta$ and $ij=-ji=k$. By taking a splitting field $L:=F(\sqrt{\eta}) \subset \R$ we embed $A$ into $M_2(L) \subset M_2(\R)$ as in Section \ref{subsection closed geodesics}: 
\begin{align}
\iota: A 
\overset{\sim}{\rightarrow} \left\{ \left.
\begin{pmatrix}
z&\eta \bar{w}\\
w&\bar{z}
\end{pmatrix}\right|
z,w \in L
\right\} \subset M_2(L); 
\begin{cases}
i &\mapsto 
\begin{pmatrix}
0&\eta \\
1&0\\
\end{pmatrix}\vspace{1mm}\\
j &\mapsto 
\begin{pmatrix}
\sqrt{\eta}&0\\
0&-\sqrt{\eta}\\
\end{pmatrix}
\end{cases}\label{iota3}
\end{align}
In the following we regard $A$ as a subalgebra of $M_2(L)\subset M_2(\R)$ via (\ref{iota3}). 
Since $\eta$ is the unique positive root of $X^3+X^2-2X-1$, we see that $A$ satisfies the condition $A\otimes_{\Q}\R\simeq M_2(\R) \times \mathbb H^{2}$. 
There is a maximal order $\mathcal O \subset A$ called the Hurwitz order which is generated (as an $\mathcal O_F$-algebra) by $i, j$ and $j':=\frac{1}{2}(1+\eta i+(1+\eta +\eta^2)j)$, i.e.,
\begin{align}
\mathcal O :=\Z[\eta][i,j,j'] \subset A.
\end{align}
Then it is known that $\Gamma_{\mathcal O}=\mathcal O^1 \subset SL_2(\R)$ becomes the $(2,3,7)$-triangle group. (Strictly speaking, the image of $\Gamma_{\mathcal O}$ in $PSL_2(\R)=\Aut(\mathfrak h)$ is the $(2,3,7)$-triangle group.) More precisely, let 
\begin{align} 
g_2 &:= ij/\eta, \label{g2}\\
g_3 &:=\frac{1}{2}(1+(\eta^2-2)j+(3-\eta^2)ij),\label{g3} \\
g_7 &:=\frac{1}{2}(\eta^2+\eta-1+(2-\eta^2)i+(\eta^2+\eta-2)ij). \label{g7}
\end{align}
Then it is known that $g_2,g_3,g_7$ are the generator of $\mathcal O^1$ with the relations $g_2^2=g_3^3=g_7^7=-1$ and $g_2=g_7g_3$. See Elkies~\cite{elkies98}, \cite{elkies99}, and Katz-Schaps-Vishne~\cite{katzschapsvishne11}.  Therefore we put $\Delta(2,3,7):=\Gamma_{\mathcal O}=\mathcal O^1$. 
See (\ref{g_2}), (\ref{g_3}), (\ref{g_7}) in Section \ref{constants} for more explicit presentation of $g_2,g_3,g_7$ as matrices in $SL_2(\R)$, and see also Remark \ref{rmk fig1} for the action of $g_2, g_3, g_7$ on the upper-half plane $\mathfrak h$.

In the next section we study the geodesics on the Shimura curve $\Delta(2,3,7)\bs \mathfrak h$ using the geodesic continued fraction.

\section{Geodesic continued fraction for $\Delta (2,3,7) \bs \mathfrak h$}\label{section GCF}

Now, we have seen in Proposition \ref{geodesic lagrange} that the geodesics $\varpi$ on $\mathfrak h$ joining special algebraic numbers become periodic on the Shimura curve $\Gamma_{\mathcal O}\bs \mathfrak h$. The geodesic continued fraction is an algorithm to observe the behavior of a given geodesic $\varpi$ on $\mathfrak h$ with respect to the action of $\Gamma_{\mathcal O}$ and enables us to detect the periodicity of $\varpi$. 

In the following we focus on the case where $\Gamma_{\mathcal O}=\Delta(2,3,7)$. Let the notations be the same as in Section \ref{ex (2,3,7)}.

\subsection{Notation}
For $z,w \in \overline{\mathfrak h}$ such that $z\neq w$, we denote by  $\brr{z,w} \subset \mathfrak h$ the open geodesic segment joining $z$ and $w$, and define by $\bsr{z,w}:= \brr{z,w} \cup \{z\}, \brs{z,w}:= \brr{z,w} \cup \{w\}, \bss{z,w}:= \brr{z,w} \cup \{z,w\} \subset \overline{\mathfrak h}$ the half open and closed geodesic segments. 
In the case where $z=w$, we assume that $\brr{z,z}=\bsr{z,z}=\brs{z,z}=\emptyset$ and $\bss{z,z}=\{z\}$. 
We also denote by $\overrightarrow{zw}$ the oriented closed geodesic segment joining $z$ to $w$, i.e., the geodesic segment $\bss{z,w}$ with orientation from $z$ to $w$. In the case where $z=w$, we assume $\overrightarrow{zz}$ has the unique trivial orientation: $z$ to $z$.

For an oriented geodesic $\varpi$ on $\mathfrak h$ and points $P,Q \in \varpi$, we introduce the natural order $\leq_{\varpi}, <_{\varpi}$ by
\begin{align}
\begin{cases}
P \leq_{\varpi} Q & \text{ if } \varpi \cap \bss{P,Q}=\overrightarrow{PQ},\\
P <_{\varpi} Q & \text{ if } \varpi \cap \bss{P,Q} = \overrightarrow{PQ} \text{ and } P \neq Q. 
\end{cases}
\end{align}

\paragraph{Fundamental domain} 

Let $\tau_2,\tau_3,\tau_7 \in \mathfrak h$ be the fixed points of the elliptic elements $g_2,g_3,g_7 \in \Delta(2,3,7)$ respectively. We also put $\tau_3' :=g_2\tau_3=g_7\tau_3 \in \mathfrak h$. Then the (closed) triangle $\mathcal F \subset \mathfrak h$ whose vertices are $\tau_3, \tau_3', \tau_7$ and whose edges are geodesic segments $\bss{\tau_3,\tau_3'},\bss{\tau_3,\tau_7},\bss{\tau_3',\tau_7}$, is known to be a fundamental domain for $\Delta(2,3,7)$. 
See \cite{katok92}*{pp.99--101}
%
%
%

Furthermore we define $\mathcal D:=\bigcup_{i=0}^6 g_7^i\mathcal F$ to be the regular geodesic heptagon with the center $\tau_7$. We denote by $\mathbf e_0:=\bsr{\tau_3,\tau_3'}$, $\mathbf e_0':=\brs{\tau_3,\tau_3'}$ the uppermost half open edges of $\mathcal D$, and define $\mathbf e_i:=g_7^i\mathbf e_0$, $\mathbf e_i':=g_7^i\mathbf e_0'$ for $i \in \Z/7\Z$. (Note that $g_7^7=-1$ acts trivially on $\mathfrak h$.) 
We denote by $\mathcal F^{\circ}$ (resp. $\mathcal D^{\circ}$) the interior of $\mathcal F$ (resp. $\mathcal D$).

We define $c_0:=\varpi_{-\sqrt{\eta} \to \sqrt{\eta}}$ to be the oriented geodesic joining $-\sqrt{\eta}$ to $\sqrt{\eta}$. 
By the explicit computation using (\ref{g_2}), (\ref{g_3}), (\ref{g_7}),
we see that $c_0$ is exactly the geodesic containing the edge $\mathbf e_0$. We denote by $S_0:=\{w \in \C \mid |w| \leq \sqrt{\eta}, \ \im w \geq 0\} \subset \overline{\mathfrak h}$ the closed semicircle ``inside'' $c_0$. Here $|w|$ is the usual Euclidean absolute value on $\C$.
See Figure \ref{FD(2,3,7)}.

\begin{figure}[hbtp]
\centering
    \includegraphics[clip,height=8cm, angle=0]{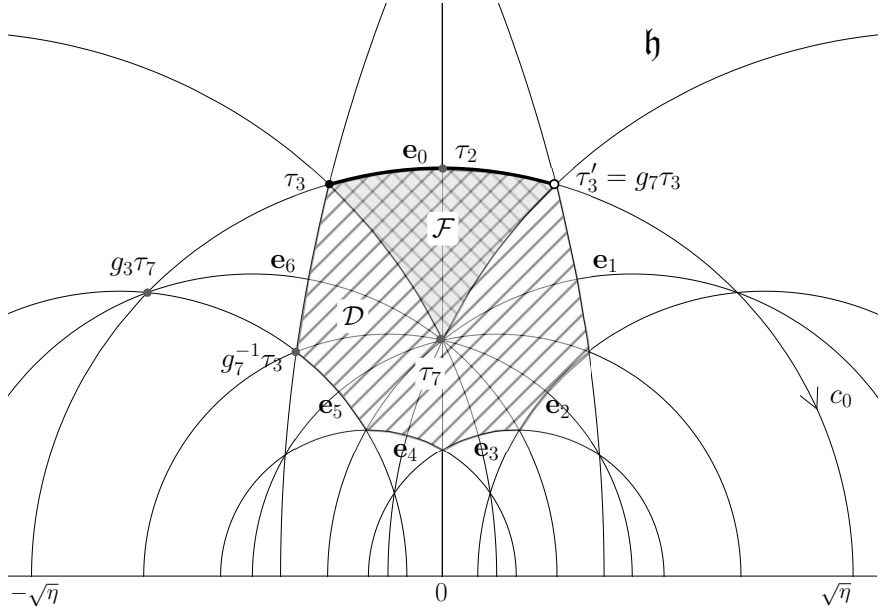} 
    \caption{Fundamental domain}\label{FD(2,3,7)}
\end{figure}

\begin{rmk}\label{rmk fig1}
In Figure \ref{FD(2,3,7)}, $g_2$ acts on $\mathfrak h$ as a rotation by $-\pi$ around $\tau_2$ (with respect to the hyperbolic metric on $\mathfrak h$),  $g_3$ acts as a rotation by $-\frac{2\pi}{3}$ around $\tau_3$, and $g_7$ acts as a rotation by $-\frac{2\pi}{7}$ around $\tau_7$. These facts can be verified by using (\ref{g_2}), (\ref{g_3}), (\ref{g_7}).
\end{rmk}

\subsection{Geodesic continued fraction algorithm}

Let $\varpi=\varpi_{\beta \to \alpha}$ be an oriented geodesic on $\mathfrak h$ joining $\beta$ to $\alpha$ ($\alpha, \beta \in \R \cup\{\infty\} $, $\alpha\neq \beta$). Note that if $\varpi \cap \mathcal D\neq \emptyset$, then there exist $P,Q \in \varpi$ such that $\varpi \cap \mathcal D= \overrightarrow{PQ}$ (possibly $P=Q$) because $\mathcal D$ is geodesically convex. 

\begin{dfn}\label{dfn reduced}
\begin{enumerate}[$(1)$]
\item We say that $\varpi$ \textit{enters (resp. leaves) $\mathcal D$ from $\mathbf e_i$ (resp. $\mathbf e_i'$)} if $\varpi \cap \mathcal D= \overrightarrow{PQ}$ for $P,Q\in \mathfrak h$ with $P \in \mathbf e_i$ (resp. $Q \in \mathbf e_i'$). 
\item We say that $\varpi$ is \textit{reduced} if $\varpi$ enters $\mathcal D$ from $\mathbf e_0$ and $|\alpha| < \sqrt{\eta}$. 
\end{enumerate}
\end{dfn}

\begin{rmk}\label{rmk reduced}
\begin{enumerate}[$1.$]
\item The reduced oriented geodesics can be classified into three types according to the way they intersect with $\mathcal D$. See Figure \ref{fig reduced} and Lemma \ref{lem trivial}. 
\item The above definition of the reducedness of geodesics is an analogue of the reducedness of real quadratic irrationals or quadratic forms in the classical theory of continued fraction. See Remark \ref{rmk sl2}.
\end{enumerate}
\end{rmk}

\begin{figure}[hbtp]
\centering
    \includegraphics[clip,height=6cm]{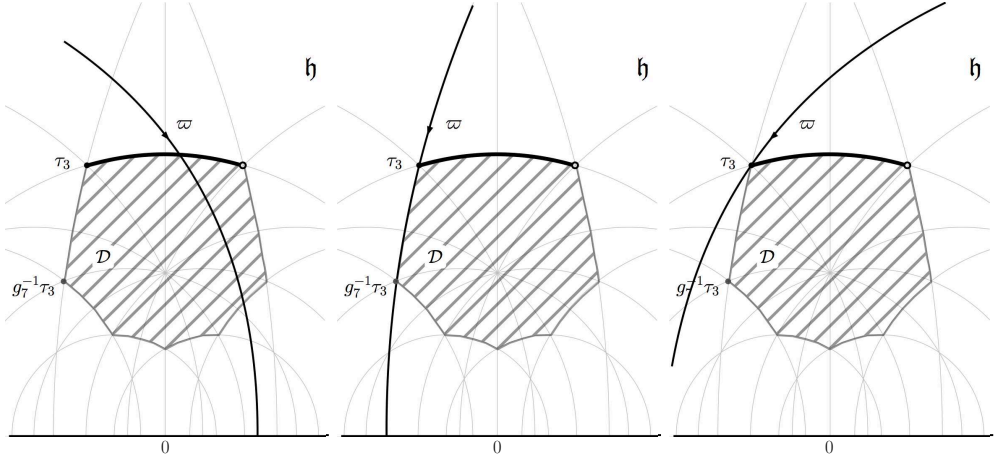} 
    \caption{$3$ types of reduced geodesics}\label{fig reduced}
\end{figure}

\begin{lem}\label{lem trivial}
Let $\varpi=\varpi_{\beta \to \alpha}$ be an oriented geodesic on $\mathfrak h$ which enters $\mathcal D$ from $\mathbf e_0$ and leaves $\mathcal D$ from $\mathbf e_i'$ $(i \in \Z/7\Z)$. Then we have the following: 
\begin{enumerate}[$(1)$]
\item If $\varpi \cap \mathcal D^{\circ} \neq \emptyset$, then we have $i \neq 0$, and both $\varpi$ and $(g_7^i g_2)^{-1}\varpi$ are reduced. 

\item If $\varpi\cap \mathcal D=\overrightarrow{\tau_3 (g_7^{-1}\tau_3)}$, then we have $i=5$ and $\varpi= g_3c_0$. In particular, we see that both $\varpi$ and $(g_7^5 g_2)^{-1}\varpi$ are reduced, and that $(g_7^5g_2)^{-1}\varpi \cap \mathcal D^{\circ} \neq \emptyset$. 

\item If $\varpi\cap \mathcal D=\overrightarrow{\tau_3 (g_7\tau_3)}$, then we have $i=0$ and $\varpi=c_0$. In particular, we see that $\varpi$ is not reduced and $g_3\varpi$ is reduced. 

\item If $\varpi\cap \mathcal D=\{\tau_3\}$, then we have $i=6$, and $(g_7^6g_2)^{-1}\varpi \cap \mathcal D^{\circ} = g_3^{-1}\varpi \cap \mathcal D^{\circ} \neq \emptyset$. 
Moreover, in this case, $\varpi$ is reduced if and only if $(g_7^6g_2)^{-1}\varpi$ is reduced.  
\end{enumerate}
As a result, for any reduced oriented geodesic $\varpi$, we see that there exists a unique index $i \in \Z/7\Z$, $i \neq 0$ such that $\varpi$  leaves $\mathcal D$ from $\mathbf e_i'$. Moreover, for such $i$, $(g_7^ig_2)^{-1}\varpi$ is again reduced. \label{recursive}
\end{lem}
\begin{proof}
Suppose $\varpi\cap\mathcal D=\overrightarrow{PQ}$ wth $P \in \mathbf e_0$ and $Q \in \mathbf e_i'$. 

(1) First, if $i=0$, then we must have $P=\tau_3$, $Q=g_7\tau_3$, and hence $\varpi \cap \mathcal D^{\circ}=\emptyset$, which is a contradiction. Therefore, $i \neq 0$. 
Take $z \in \varpi \cap \mathcal D^{\circ} \subset S_0$. We have $P \in c_0$, $z \notin c_0$, and $z \in \brr{P, Q} \subset  \brr{P, \alpha}$. If $|\alpha| \geq \sqrt{\eta}$, then we have either $\brr{P, \alpha} \subset c_0$ or $\brr{P, \alpha} \subset \mathfrak h - S_0$, which is a contradiction. Therefore, $\varpi$ is reduced. 
Next, set $\varpi':= (g_7^ig_2)^{-1}\varpi$, $P':=(g_7^ig_2)^{-1}P$, and $Q':=(g_7^ig_2)^{-1}Q$. Then we see that $Q' \in \varpi' \cap \mathbf e_0$, $\brr{P',Q'} \subset \mathfrak h - S_0$. These imply that $\varpi'$ enters $\mathcal D$ from $\mathbf e_0$ and $|(g_7^ig_2)^{-1}\alpha| < \sqrt{\eta}$, and hence $\varpi'$ is reduced. 

The assertions (2), (3) and the first half of (4) are clear. 
To see the latter half of (4), observe that (under the assumption $\varpi \cap \mathcal D=\{\tau_3\}$) $\varpi$ is reduced if and only if $-\sqrt{\eta} < \alpha < g_3 \sqrt{\eta}$, where $g_3 \sqrt{\eta}$ is the linear fractional transformation of $\sqrt{\eta}$ by $g_3$. 
Then we further see that this is equivalent to $(g_7^6g_2)^{-1}\varpi$ being reduced. 

The last assertion follows directly from (1) to (4). 
\end{proof}

\begin{lem}\label{lem reduced}
\begin{enumerate}[$(1)$]
\item For any oriented geodesic $\varpi$ and $z \in \varpi$, there exists $\gamma \in \Delta(2,3,7)$ such that $\gamma \varpi$ is reduced and $\gamma z\in \mathcal D$. \label{reducibility1}

\item For any oriented geodesic $\varpi$, there exists $\gamma \in \Delta(2,3,7)$ such that $\gamma \varpi$ is reduced and $\gamma\varpi \cap \mathcal D^{\circ}\neq \emptyset$. \label{reducibility2}

\item Let $\varpi$ be a reduced oriented geodesic such that $\varpi \cap \mathcal D^{\circ} \neq \emptyset$ and let $z \in \varpi\cap\mathcal D^{\circ}$. 
Suppose $\gamma \varpi$ is reduced and $\gamma z\in \mathcal D$ for $\gamma \in \Delta(2,3,7)$. Then we have $\gamma = \pm 1$. \label{torsion}
\end{enumerate}
\end{lem}

\begin{proof}
(1) Since $\mathcal F \subset \mathcal D$, where $\mathcal F$ is the fundamental domain, there exists $\gamma' \in\Delta(2,3,7)$ such that $\gamma' z \in \gamma' \varpi \cap \mathcal D=\overrightarrow{PQ}$ for some $P, Q \in \mathfrak h$. Suppose $P\in \mathbf e_i$, and set $\varpi ':= g_7^{-i}\gamma' \varpi$, $z' :=g_7^{-i}\gamma' z \in \mathcal D$. Then enters $\mathcal D$ from $\mathbf e_0$. 
If $\varpi' \cap \mathcal D^{\circ} \neq \emptyset$ or $\varpi' \cap \mathcal D=\overrightarrow{\tau_3 (g_7^{-1}\tau_3)}$, then by Lemma \ref{lem trivial} (1), (2), $\varpi'$ is reduced as desired. 
If $\varpi' \cap \mathcal D=\overrightarrow{\tau_3 (g_7\tau_3)}$, then by Lemma \ref{lem trivial}~(3), we have $\varpi'=c_0$, and hence we see that $g_3 \varpi'$ is reduced and $g_3 z' \in \mathcal D$. 
Otherwise, we have $\varpi' \cap \mathcal D=  \{\tau_3\}$ and $z'=\tau_3$. In this case we easily see that either $\varpi'$ is reduced or $g_3\varpi'$ is reduced and $g_3 z' \in \mathcal D$.  

(2) This follows from (1) and Lemma \ref{lem trivial}. Indeed, by (1), we can find $\gamma \in \Delta(2,3,7)$ such that $\gamma \varpi$ is reduced. Then by replacing $\gamma$ by $(g_7^5g_2)^{-1}\gamma$ or $(g_7^6g_2)^{-1}\gamma$ if necessary, we further obtain $\gamma \varpi \cap \mathcal D^{\circ} \neq \emptyset$. 

(3) Suppose $\varpi\cap\mathcal D=\overrightarrow{PQ}$ and $\gamma z \in g_7^{i} \mathcal F$ ($i \in \Z/7\Z$). Then since $ z \in \gamma^{-1} g_7^{i} \mathcal F \cap \mathcal D^{\circ}$, there exists $j \in \Z/7\Z$ such that $\gamma = \pm g_7^{j}$. 
In particular, we see that $\gamma \varpi \cap \mathcal D= \gamma \varpi \cap \gamma \mathcal D= \overrightarrow{\gamma P \gamma Q}$. 
On the other hand, since $\varpi$ and $\gamma \varpi$ are both reduced, we have $P \in \mathbf e_0$ and $\gamma P \in \mathbf e_0$. Therefore, we see that $j=0$, and hence $\gamma = \pm 1$. 
\end{proof}

Now we define the geodesic continued fraction algorithm following the general principle of Morse~\cite{morse21}, Series~\cite{series85}, Katok~\cite{katok96}. Note that we slightly modify the original algorithm by using $\mathcal D$ instead of $\mathcal F$. 
\begin{dfn}[Geodesic continued fraction algorithm for $\Delta(2,3,7)\bs \mathfrak h$]\label{GCF algo}
Let $\varpi$ be an oriented geodesic on $\mathfrak h$. Define $B_0\in \Delta(2,3,7)$ and $i_k \in \Z/7\Z\!-\!\{0\}$ ($k=1,2,3,\dots$) by the following algorithm:
\begin{itemize}
\item Find (any) $B_0\in \Delta(2,3,7)$ such that $B_0^{-1}\varpi$ is reduced, and set $\varpi_0:=B_0^{-1}\varpi$. 
\item For a given reduced oriented geodesic $\varpi_k$ ($k \geq 0$), find the unique $i_{k+1} \in \Z/7\Z\!-\!\{0\}$ such that $\varpi_k$ leaves $\mathcal D$ from $\mathbf e_{i_{k+1}}'$. Set $\varpi_{k+1}:=(g_7^{i_{k+1}}g_2)^{-1}\varpi_k$. Then by Lemma \ref{lem trivial}, $\varpi_{k+1}$ is again reduced.
\end{itemize} 
We call this the \textit{geodesic continued fraction expansion} of $\varpi$ (with respect to $\Delta(2,3,7)$), and express it as
\begin{align}
\varpi = \llbracket B_0; i_1, i_2, i_3, \cdots \rrbracket_{\Delta(2,3,7)}\text{ or } 
B_0^{-1}\varpi = \llbracket i_1, i_2, i_3, \cdots \rrbracket_{\Delta(2,3,7)}. \label{GCF exp algo}
\end{align}
\end{dfn}

We review some basic properties of geodesic continued fraction expansion.  
See also Figure \ref{fig conv}.
\begin{prop}\label{prop GCF}
Let $\varpi$ be an oriented geodesic on $\mathfrak h$, and let 
\begin{align}
\varpi &= \llbracket B_0; i_1, i_2, i_3, \cdots \rrbracket_{\Delta(2,3,7)} 
\end{align} 
be the geodesic continued fraction expansion of $\varpi$. 
\begin{enumerate}[$(1)$] 
\item The choice of $B_0$ is not unique. However once we choose $B_0$, then the sequence $i_1, i_2, \dots$ are uniquely determined. More generally, let $\varpi'$ be another oriented geodesic (possibly $\varpi'=\varpi$), and let 
\begin{align}
\varpi' = \llbracket B_0'; j_1, j_2, j_3, \cdots \rrbracket_{\Delta(2,3,7)},
\end{align}
be the geodesic continued fraction expansion of $\varpi'$ such that $(B_0')^{-1}\varpi'=B_0^{-1}\varpi$, then we have $j_k=i_k$ for all $k\geq 1$.
\item For $k\geq 1$, set $A_k:=g_7^{i_k}g_2 \in \Delta(2,3,7)$ and $B_k:=B_0A_1A_2\cdots A_k \in \Delta(2,3,7)$ so that $\varpi_k=B_k^{-1}\varpi$ in the algorithm. Moreover define $P_k, Q_k \in \varpi$ so that $\varpi \cap B_k \mathcal D= \overrightarrow{P_kQ_k}$. Then we have: 
\begin{enumerate}[{\rm (i)}]
\item $P_k \in B_k \mathbf e_0$, and $Q_k=P_{k+1} \in B_k \mathbf e_{i_{k+1}}'$. In particular, we have $P_k \leq_{\varpi} P_{k+1}$.
\item $P_k \neq Q_{k+1}=P_{k+2}$. In particular, we have $P_k <_{\varpi} P_{k+2}$.
\item $B_k \mathcal D \neq B_l \mathcal D$ for $k\neq l$.
\end{enumerate}
\end{enumerate}
\end{prop}
\begin{proof}
(1) and (2) {\rm (i)} are clear from Lemma \ref{lem trivial}, and Definition \ref{GCF algo}. 

(2) {\rm (ii)} Note that by (i) we have $P_k \leq_{\varpi} Q_k=P_{k+1} \leq_{\varpi} Q_{k+1}$ in general. 
Suppose $P_k=Q_{k+1}$. 
Then we have $P_k=Q_k=P_{k+1}=Q_{k+1}$, and hence $P_k=Q_k=B_{k} \tau_3$ and $P_{k+1}=Q_{k+1}=B_{k+1}\tau_3$. 
Therefore, it follows that $A_{k+1} = g_7^{-1}g_2 = g_3$. Thus we get $B_k^{-1} \varpi \cap \mathcal D=g_3^{-1}B_k^{-1} \varpi \cap \mathcal D=\{\tau_3\}$, which is impossible by Lemma \ref{lem trivial} (4).
%

(2) {\rm (iii)} Suppose $B_k \mathcal D = B_l \mathcal D$ for $k\leq l$. Since $\mathcal D$ is geodesically convex, 
we have $P_k=P_l$ and $Q_k=Q_l$. Therefore, by (i) we obtain $P_k=P_m=Q_m=Q_l$ for all $k \leq m \leq l$. Then by {\rm (ii)}, we have $k=l$.
\end{proof}

Suppose that $\varpi$ joins $\beta$ to $\alpha$ ($\alpha, \beta \in \R\cup\{\infty\}$, $\alpha \neq \beta$), i.e., $\varpi=\varpi_{\beta \to \alpha}$. 
Suppose also that $\varpi$ is reduced for simplicity, and hence we take $B_0=1$ in the algorithm. Let 
\begin{align}
\varpi =\varpi_{\beta \to \alpha}= \llbracket  i_1, i_2, i_3, \dots \rrbracket_{\Delta(2,3,7)}
\end{align}
be the geodesic continued fraction expansion of $\varpi$. For $k\geq 1$, set $A_k:=g_7^{i_k}g_2 \in \Delta(2,3,7)$ and $B_k:=A_1A_2\cdots A_k \in \Delta(2,3,7)$ as in Proposition \ref{prop GCF}. 
Then by the definition of the algorithm, the sequence $B_k\mathcal D$ ($k=1,2,3,\dots$) (of subsets of $\mathfrak h$) seems to ``approach'' to $\alpha$ as $k$ goes to $\infty$. See Figure \ref{fig conv}. In fact we can prove 
\begin{align}\label{limit tau7}
\alpha = \lim_{k \to \infty} B_k \tau_7.
\end{align}
See Theorem \ref{thm convergence}.

Now, note that for 
$\gamma=
\begin{pmatrix}
a&b\\
c&d\\
\end{pmatrix} \in SL_2(\R)
$ such that $c\neq 0$, we can rewrite the linear fractional transformation $\gamma z$ ($z \in \mathfrak \P^1(\C)$) as
\begin{align}
\gamma z=
\frac{az+b}{cz+d}
=
\frac{a}{c}-\frac{1/c^2}{d/c+z}.
\end{align}
Therefore for $i \in \Z/7\Z\!-\!\{0\}$ we can define $\bm a_i, \bm b_i, \bm c_i \in L=\Q(\sqrt{\eta})$ by 
\begin{align}
g_7^ig_2 z=:
\bm a_i-\frac{\bm b_i}{\bm c_i+z}
\end{align}
because we easily see that the lower left component of $g_7^ig_2$ is non-zero for $i \neq 0$. 
In Section \ref{constants} we compute these constants $\bm a_i, \bm b_i, \bm c_i$ explicitly. 
Then we can \underline{formally} rewrite (\ref{limit tau7}) as
\begin{align}\label{CF}
\alpha=\bm a_{i_1}-\cfrac{\bm b_{i_1}}{\bm c_{i_1}+\bm a_{i_2}-\cfrac{\bm b_{i_2}}{\bm c_{i_2}+\bm a_{i_3}-\cfrac{\bm b_{i_3}}{\bm c_{i_3}+\bm a_{i_4}-\cfrac{\bm b_{i_4}}{\bm c_{i_4}+\cdots}}}}
\end{align}
In the following section we study the convergence of this continued fraction expansion of $\alpha$.

\begin{rmk}[Remark on the case of $SL_2(\Z)$]\label{rmk sl2}
Here we briefly explain the background of the above definitions of reducedness of a geodesic $\varpi$ and the geodesic continued fraction by comparing to the case of $SL_2(\Z)$. Notations in this remark are independent of the rest of the paper. 

As we have seen in Section \ref{sect modular curve}, the modular curve $SL_2(\Z)\bs \mathfrak h$ is the Shimura curve associated to the quaternion algebra $M_2(\Q)$ over $\Q$ and a maximal order $M_2(\Z)$. Now, $SL_2(\Z)$ is the $(2,3, \infty)$-triangle group $\Delta(2,3,\infty)$ generated by 
\begin{align}
g_2=
\begin{pmatrix}
0&-1 \\
1&0 \\
\end{pmatrix}, \quad
g_3=
\begin{pmatrix}
-1&-1 \\
1&0 \\
\end{pmatrix}, \quad
g_{\infty}=
\begin{pmatrix}
1&1 \\
0&1 \\
\end{pmatrix}. 
\end{align}
The triangle $\mathcal F=\left\{z \in \mathfrak h \mid |z|\geq 1, -\frac{1}{2} \leq \re z \leq \frac{1}{2} \right\}$ is known to be a fundamental domain. Our regular geodesic heptagon corresponds to the $\infty$-gon $\mathcal D=\bigcup_{i \in \Z} g_{\infty}^i\mathcal F$, and our $\mathbf e_i , \mathbf e_i'$ correspond to 
\begin{align}
\mathbf e_i&=\left\{z \in \mathfrak h ~\left|~ |z|= 1, i-\frac{1}{2} < \re z \leq i+\frac{1}{2} \right. \right\} \\
\mathbf e_i'&=\left\{z \in \mathfrak h ~\left|~ |z|= 1, i-\frac{1}{2} \leq  \re z < i+\frac{1}{2} \right. \right\},
\end{align}
respectively. See Figure \ref{fig sl2}. 

\begin{figure}[hbtp]
\centering
    \includegraphics[clip,height=7cm]{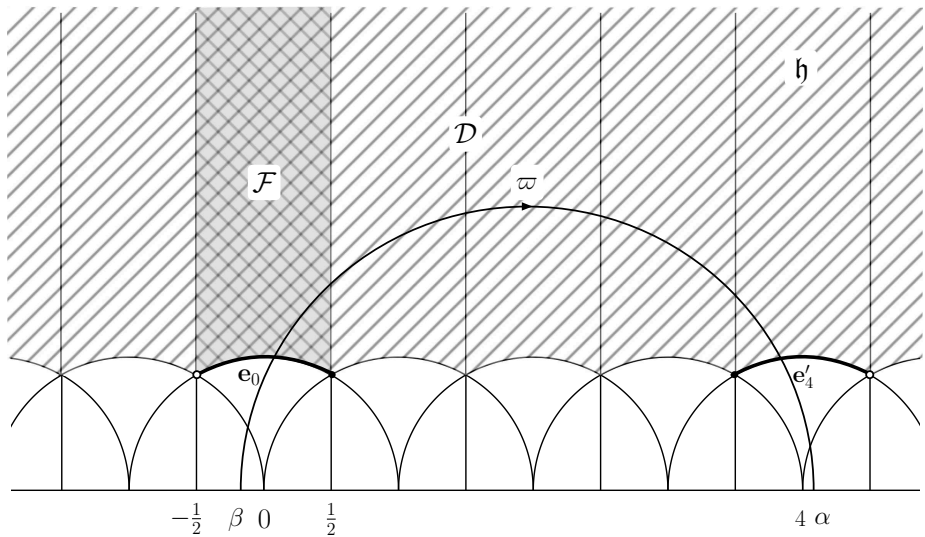} 
    \caption{The case of $SL_2(\Z)$}\label{fig sl2}
\end{figure}

Let $\varpi=\varpi_{\beta \to \alpha}$ be the oriented geodesic on $\mathfrak h$ joining $\beta$ to $\alpha$ ($\alpha, \beta \in \R$). 
Then our definition of the reduced geodesics (Definition \ref{dfn reduced}) in this case would be the following: 
$\varpi$ is said to be {\it reduced} if $\varpi$ enters $\mathcal D$ from $\mathbf e_0$ and $|\alpha|>1$. 
Then this condition can be seen roughly as:
\begin{align}
|\beta| \text{``}<\text{''} 1 \text{ and } |\alpha| > 1.
\end{align}
Hence we see that this definition is an analogue of the reducedness of real quadratic irrationals (cf. \cite{zagier75}): a real quadratic irrational $\alpha$ is said to be {\it reduced} if $\alpha$ and its conjugate $\alpha'$ (over $\Q$) satisfy
\begin{align}
0<\alpha'<1 \text{ and } \alpha>1. 
\end{align}
We also see that our geodesic continued fraction expansion of $\varpi$ in this case of $SL_2(\Z)$ coincides with the one called the ``geometric code'' and studied in \cite{katok96} by Katok. 
Furthermore since we have 
$
g_{\infty}^ig_2z =i-\frac{1}{z},
$ 
(\ref{CF}) becomes a variant of the so called ``$-$''-continued fraction expansion. 
\end{rmk}


\subsection{Convergence}\label{subsection convergence}
Perhaps the most traditional way to discuss the convergence of the above continued fraction (\ref{CF}) is to consider the limit of the following $k$-th convergent:
\begin{align}
\bm x_k^{trad}&= \bm a_{i_1}-\cfrac{\bm b_{i_1}}{\bm c_{i_1}+\bm a_{i_2}-\cfrac{\bm b_{i_2}}{\bm c_{i_2}+\cdots-\cfrac{\bm b_{i_k}}{\bm c_{i_k}+\bm a_{i_{k+1}}}}} \label{conv trad}\\
&=B_k\bm a_{i_{k+1}} = B_{k+1}\infty, \label{conv trad2}
\end{align}
Here $B_k \bm a_{i_{k+1}}$ and $B_{k+1}\infty$ are the linear fractional transformations of $\bm a_{i_{k+1}}$ and $\infty$.
Unfortunately, however, we can give an example in which the traditional $k$-th convergent $\bm x_k^{trad}$ does not converge to $\alpha$. See Example \ref{ex0}.

Instead, we consider the following ``regularized'' $k$-th convergent:
\begin{align}
\bm x^{reg}_k =B_k0= \bm a_{i_1}-\cfrac{\bm b_{i_1}}{\bm c_{i_1}+\bm a_{i_2}-\cfrac{\bm b_{i_2}}{\bm c_{i_2}+\cdots-\cfrac{\bm b_{i_k}}{\bm c_{i_k}}}}, \label{conv GCF}
\end{align}
which seems more natural in our setting since this $\bm x^{reg}_k$ corresponds to exactly the first $k$ steps of the geodesic continued fraction expansion of $\varpi$.


\begin{dfn} 
Let $\bm i = (i_k)_{k \geq 1}$ ($i_k \in \Z/7\Z\!-\!\{0\}$) be a sequence.
\begin{enumerate}[$(1)$]
\item  We define the \textit{associated formal continued fraction} $\bm x(\bm i)$ by the right hand side of (\ref{CF})
\item We define the \textit{associated traditional $k$-th convergent} $\bm x^{trad}_k(\bm i)$ by the right hand side of (\ref{conv trad}). If the traditional $k$-th convergent $\bm x^{trad}_k(\bm i)$ converges to $x \in \P^1(\C)=\C \cup\{\infty\}$ with respect to the natural topology of $\P^1(\C)$, we say that the continued fraction $\bm x(\bm i)$ converges to $x$ in the traditional sense.
\item We define the \textit{associated regularized $k$-th convergent} $\bm x^{reg}_k(\bm i)$ by the right hand side of (\ref{conv GCF}). If the regularized $k$-th convergent $\bm x^{reg}_k(\bm i)$ converges to $x \in \P^1(\C)=\C\cup \{\infty\}$ with respect to the natural topology of $\P^1(\C)$, we say that the continued fraction $\bm x(\bm i)$ converges to $x$ in the regularized sense.
\end{enumerate}
\end{dfn}

Recall that $S_0=\{w \in \C \mid |w| \leq \sqrt{\eta}, \ \im w \geq 0\} \subset \overline{\mathfrak h}$ is the closed semicircle inside the geodesic $c_0$.  
We prove the following. 
\begin{thm}\label{thm convergence}
Let $\varpi=\varpi_{\beta \to \alpha}$ be an oriented geodesic on $\mathfrak h$ joining $\beta$ to $\alpha$, and let 
\begin{align}
B_0^{-1}\varpi= \llbracket i_1,i_2,i_3, \cdots \rrbracket_{\Delta(2,3,7)}
\end{align}
be the geodesic continued fraction expansion of $\varpi$. For $k \geq 1$, set $A_k:=g_7^{i_k}g_2 \in \Delta(2,3,7)$ and $B_k:=B_0A_1A_2\cdots A_k \in \Delta(2,3,7)$. 
We denote by $B_kS_0 \subset \overline{\mathfrak h}$ the $B_k$-translation of $S_0$. 
Then for any sequence $(z_k)_{k \geq 0}$ such that $z_k \in B_kS_k$, we have
\begin{align}
\lim_{k \to \infty} z_k =\alpha 
\end{align}
with respect to the natural topology of $\P^1(\C)$. 
In particular, for any $z \in S_0$ we obtain 
\begin{align}
\lim_{k \to \infty} B_k z=\alpha.
\end{align}
\end{thm}

\begin{cor}\label{cor convergence}
The associated formal continued fraction $\bm x((i_1, i_2, \dots))$ converges to $B_0^{-1}\alpha$ in the regularized sense, i.e., 
\begin{align}
\lim_{k \to \infty} \bm x^{reg}_k((i_1, i_2, \dots))=B_0^{-1}\alpha
\end{align}
\end{cor}
\begin{proof}
This follows from $0 \in S_0$ and $B_k0=B_0A_1\cdots A_k0=B_0\bm x^{reg}_k((i_1,i_2,\dots))$.
\end{proof}

By the explicit computation of $\bm a_i, \bm b_i, \bm c_i$ in Section \ref{constants}, we have $|\bm a_1|=|\bm a_{-1}|<|\bm a_2|=|\bm a_{-2}|<\sqrt{\eta} < |\bm a_3|=|\bm a_{-3}|$, cf. (\ref{ai value}). From this and Theorem \ref{thm convergence}, we can also say a little bit about the convergence in the traditional sense. 
\begin{cor} \label{conv cor2}
We keep the notations in Theorem \ref{thm convergence}
\begin{enumerate}[$(1)$]
\item Let $(k_l)_{l \geq 1} $ be the subsequence of $(k)_{k\in \Z_{\geq 1}}$ consisting of those $k\geq 1$ such that $i_k\neq 3,4$ in $\Z/7\Z$. Then we have 
\begin{align}
\lim_{l\to \infty}\bm x^{trad}_{k_l}((i_1,i_2, \dots))=B_0^{-1}\alpha. 
\end{align}
\item In particular, if $i_k\neq 3,4$ in $\Z/7\Z$ for all sufficiently large $k \geq 1$, then the associated formal continued fraction $\bm x(\bm i)$ converges to $B_0^{-1}\alpha$ in the traditional sense.
\end{enumerate}
\end{cor}

\paragraph{Proof of the convergence}
Here we give a proof of Theorem \ref{thm convergence}. 

Put $\Gamma:=\Delta(2,3,7)$ for simplicity. 
We may assume that $\varpi=\varpi_{\beta\to \alpha}$ is reduced and $B_0=1$. 
Recall that $c_0=\varpi_{-\sqrt{\eta} \to \sqrt{\eta}}$ is the oriented geodesic on $\mathfrak h$ which contains the edge $\mathbf e_0$.
We denote by $u_0:=\sqrt{\eta}$, $v_0:=-\sqrt{\eta}$ the two endpoints of $c_0$. 
For $k\geq 0$ we define $c_k:=B_kc_0$, $u_k:=B_ku_0$, $v_k:=B_kv_0$, $S_k:=B_kS_0$ to be the $B_k$-translations of the corresponding objects. 
By the definition of the geodesic continued fraction algorithm, $B_k^{-1}\varpi$ is reduced. Thus we define $P_k, Q_k \in \varpi$ ($k\geq 0$) so that $\varpi \cap B_k\mathcal D=\overrightarrow{P_kQ_k}$. Then by Proposition \ref{prop GCF} (2) {\rm (i)}, we have $P_k \in \varpi \cap c_k$ and $Q_k=P_{k+1} \in \varpi \cap c_{k+1}$. See Figure \ref{fig conv}.

\begin{figure}[hbtp]
\centering
    \includegraphics[clip,height=9cm]{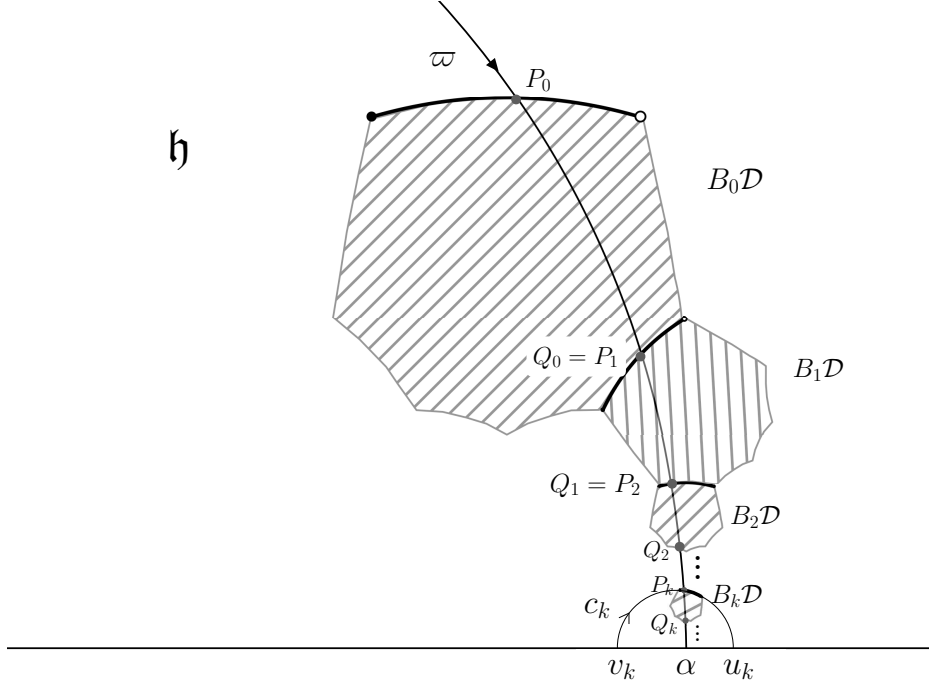} 
    \caption{Geodesic continued fraction algorithm of $\varpi$}\label{fig conv}
\end{figure}


In the following we prepare six technical lemmas, most of which are intuitively clear.  
We adopt the usual notation of intervals in $\P^1(\R)=\R \cup \{\infty\}$ by defining
\begin{align}
(u,v)=
\begin{cases}\label{extended interval}
(u,v) &\text{ if } u, v \in \R \text{ and } u<v\\
(u,\infty)\cup\{\infty\} \cup (-\infty, v) &\text{ if } v, u \in \R \text{ and } v<u\\
(u,\infty) &\text{ if } u \in \R, v=\infty\\
(-\infty,v) &\text{ if } v \in \R, u=\infty,
\end{cases}
\end{align} 
for $u,v \in \R \cup \{\infty\}$ such that $u\neq v$.
\begin{lem}\label{pf lem0}
For all $k \geq 0$ we have $\alpha \in (v_k, u_k)$ and $\beta \in (u_k, v_k)$.
\end{lem}
\begin{proof}
By the definition of the geodesic continued fraction algorithm, $B_k^{-1}\varpi$ is reduced, and hence we see $B_k^{-1}\alpha \in (v_0,u_0)$ and $B_k^{-1}\beta \in (u_0, v_0)$. Now the lemma follows from the fact that $SL_2(\R)$ action preserves the intervals in $\R \cup \{\infty\}$, i.e., $g(u,v)=(gu,gv)$ for $g\in SL_2(\R)$ and $u, v \in \R\cup \{\infty\}$ such that $u\neq v$.
\end{proof}
Next we observe the behavior of $c_k$ as $k$ goes to $\infty$, or more generally, the behavior of $\Gamma$-translations of $c_0$. 
\begin{lem}\label{pf lem1}
\begin{enumerate}[$(1)$]
\item The projection $\overline{c_0}$ becomes a closed geodesic on $\Gamma \bs \mathfrak h$, and there exists a hyperbolic element $\gamma_0 \in \Gamma$ such that $\Gamma_{c_0} (=\{\gamma \in \Gamma \mid \gamma c_0=c_0\})=\{\pm \gamma_0^{k} \mid k\in \Z\}$.
\item In particular, we can decompose $c_0$ into a disjoint union of segments as
\begin{align}
c_0 = \coprod_{k \in \Z}\gamma_0^k \bsr{\tau_3, \gamma_0\tau_3}.
\end{align}
\end{enumerate}
\end{lem}
\begin{proof}
(1) follows from Proposition \ref{geodesic lagrange} and Lemma \ref{lem ass field} (5), (6). Indeed by choosing $z=w=1$, we obtain $D_{z,w}=(z-\bar z)^2 + 4\eta w\bar w=4\eta$ and $\pm \sqrt{\eta} = \frac{1}{2w}((z-\bar z) \pm \sqrt{D_{z,w}})$. 
(2) follows from (1) and $\tau_3\in c_0$.
\end{proof}

We consider the set 
\begin{align}
\mathbb I :=\{\gamma c_0 \mid \gamma \in \Gamma, \#(\gamma c_0 \cap c_0) =1\}
\end{align}
of all $\Gamma$-translations of $c_0$ which intersect with $c_0$ at one point in $\mathfrak h$. Note that if two geodesics on $\mathfrak h$ has an intersection, then either they intersect at one point or they coincide up to the orientation. 
\begin{lem}\label{pf lem2}
There exist $c^{(1)}, \dots, c^{(r)} \in \mathbb I$ (for some $r\geq 0$) such that 
\begin{align}
\mathbb I = \bigcup_{k \in \Z} \{\gamma_0^k c^{(1)}, \dots, \gamma_0^k c^{(r)} \}.
\end{align}
If $\mathbb I=\emptyset$, then we assume $r=0$ and the both sides are the empty set.
\end{lem}
\begin{proof}
First, since $\Gamma=\Delta(2,3,7)$ acts properly discontinuously on $\mathfrak h$, we have
\begin{align}
\# \{\gamma \in \Gamma \mid \#(\gamma \bss{\tau_3, \gamma_0\tau_3} \cap \bss{\tau_3, \gamma_0\tau_3})=1 \} < \infty
\end{align}
Put $\{\gamma^{(1)}, \dots \gamma^{(r)}\}:=\{\gamma \in \Gamma \mid \#(\gamma \bss{\tau_3, \gamma_0\tau_3} \cap \bss{\tau_3, \gamma_0\tau_3})=1 \}$, and set $c^{(i)}:=\gamma^{(i)}c_0 \in \mathbb I$ ($i=1, \dots, r$).
Now, take any $\gamma c_0 \in \mathbb I$. By Lemma \ref{pf lem1} (2) there exists $k, l \in \Z$ such that
\begin{align}
\#(\gamma \gamma_0^l \bsr{\tau_3, \gamma_0 \tau_3} \cap \gamma_0^k \bsr{\tau_3, \gamma_0 \tau_3})=1.
\end{align}
Then we get $\gamma= \gamma_0^{k}\gamma^{(i)}\gamma_0^{-l}$ for some $i\in \{1, \dots r\}$, and hence $\gamma c_0= \gamma_0^{k}c^{(i)}$. 
\end{proof}
For $i=1, \dots, r$, we denote by $u^{(i)} , v^{(i)} \in \R\cup \{\infty\}$ the two end points of $c^{(i)}$ such that $c^{(i)}=\varpi_{v^{(i)}\to u^{(i)}}$.
\begin{lem}\label{pf lem3}
For any $\epsilon>0$ there exists $N>0$ such that for any $k \in \Z$, $i \in \{1,\dots, r\}$ and $z \in \gamma_0^kc^{(i)}$ with $|k|>N$, we have either $|z-u_0|<\epsilon$ or $|z-v_0|<\epsilon$. 
\end{lem}
\begin{proof}
Now $\gamma_0$ is a hyperbolic element with fixed points $u_0$ and $v_0$. We may assume $u_0$ is the attracting point. On the other hand, we have $u^{(i)}, v^{(i)} \notin \{u_0, v_0\}$ because $\# (c^{(i)}\cap c_0)=1$. Therefore we get $\lim_{k\to \infty} \gamma_0^k u^{(i)}=\lim_{k\to \infty}\gamma_0^k v^{(i)}=u_0$ and $\lim_{k\to -\infty}\gamma_0^k u^{(i)}=\lim_{k\to -\infty}\gamma_0^k v^{(i)}=v_0$. This proves the lemma.
\end{proof}

\begin{lem}\label{pf lem3.5}
\begin{enumerate}[$(1)$]
\item For $k \neq l$ we have $c_k\neq c_l$ as subsets of $\mathfrak h$. 
\item For any fixed $l \geq 0$, we have $P_k \in S_l$ for all $k \geq l$. 
\item For any fixed $k \geq 0$, we have $P_k \notin S_l$ for all $l \geq k+2$.
\end{enumerate}
\end{lem}
\begin{proof}
(1) Suppose $c_k=c_l$ for $k \leq l$. Then we have $P_k=P_l$, and hence $l \leq k+2$ by Proposition \ref{prop GCF} (2) {\rm (ii)}. On the other hand, by the definition of the algorithm, we easily see that $c_{k+1}\neq c_k$. Therefore, we obtain $k=l$.

(2), (3) These follow from Lemma \ref{pf lem0} and Proposition \ref{prop GCF} (2) {\rm (i), (ii)}. Indeed, since $S_l$ is geodeiscally convex and $\varpi \cap c_l=\{P_l\}$, we have $\bsr{P_l, \alpha} = \varpi \cap S_l$ by Lemma \ref{pf lem0}. On the other hand we have $P_k \in \bsr{P_l, \alpha}$ for $k \geq l$, and $P_k \notin \bsr{P_l, \alpha}$ for $l \geq k+2$ by Proposition \ref{prop GCF} (2) {\rm (i)} and {\rm (ii)} respectively. This proves the lemma. 
\end{proof}

\begin{lem}\label{pf lem4}
For any fixed $l \geq 0$ we have
\begin{align}
\# \{k\geq 0 \mid c_k \cap c_l \neq \emptyset \} <\infty.
\end{align}
\end{lem}
\begin{proof}
Suppose the contrary, i.e., there exists a subsequence $(k_n)_{n\geq 1}$ of $(k)_{k\in \Z_{\geq 0}}$ such that $c_{k_n}\cap c_l \neq \emptyset$ for all $n\geq 1$. We may assume $k_n >l$ for all $n \geq 1$. Then we see $\# (c_{k_n}\cap c_l )=1$ by Lemma \ref{pf lem3.5} (1). Therefore we have $B_l^{-1}B_{k_n}c_0 \in \mathbb I$ for all $n\geq 1$. 
By Lemma \ref{pf lem2}, for each $n$, the geodesic $B_l^{-1}B_{k_n}c_0$ can be written as $B_l^{-1}B_{k_n}c_0=\gamma_0^{m_n}c^{(i_n)}$ for some $m_n \in \Z$ and $i_n \in \{1, \dots, r\}$. 
Moreover, since $B_l^{-1}B_{k_n}c_0$ ($n \geq 1$) are distinct by Lemma \ref{pf lem3.5} (1), we see $|m_n| \to \infty$ as $n$ goes to $\infty$. 
%

On the other hand, we have $B_l^{-1}P_{k_n} \in B_l^{-1}B_{k_n}c_0 \cap B_l^{-1}\varpi$ for all $n\geq 1$. 
Therefore, by Lemma \ref{pf lem3} and Lemma \ref{pf lem3.5} (2), we obtain 
\begin{align}
B_l^{-1}\alpha= \lim_{n\to \infty}  B_l^{-1}P_{k_n} \in \{u_0, v_0\}.
\end{align}
However this contradicts to Lemma \ref{pf lem0}. 
\end{proof}

\begin{proof}[Proof of Theorem \ref{thm convergence}]
By Lemma \ref{pf lem4}, there exists $l_0 \geq 2$ such that for all $k \geq l_0$ we have $c_k \cap c_0 = \emptyset$. Then for $k \geq l_0 \geq 2$, we have either $S_0 \subset S_k$ or $S_k \subset S_0$. 
%
%
%
%
However, by Lemma \ref{pf lem3.5} (3) we have $P_0 \notin S_k$, and hence $S_0 \subset S_k$ can not happen, and we get $S_k \subsetneq S_0$ for all $k \geq l_0$. 
Then, inductively, we can find a strictly increasing sequence $(l_n)_{n\geq 0}$ such that we have $S_k \subsetneq S_{l_{n-1}}$ for all $k \geq l_n$. Indeed, for given $l_{n-1}$, there exists $l_n \geq l_{n-1}+2$ such that $c_k\cap c_{l_{n-1}} =\emptyset$ for all $k \geq l_n$ by Lemma \ref{pf lem4}. Then we have $S_k\subsetneq S_{l_{n-1}}$ for all $k \geq l_n \geq l_{n-1}+2$ by Lemma \ref{pf lem3.5} (3). 
In order to prove the theorem, it suffices to show
\begin{align}
\bigcap_{n\geq 0} S_{l_n}=\{\alpha\}. \label{pf claim}
\end{align}
We denote by $S_{\infty}$ the left hand side of (\ref{pf claim}). Since $S_k$ are all closed semicircle, the intersection $S_{\infty}$ is also a closed semicircle, i.e., there exists $\xi \in \R (\cap S_0)$ and $\lambda\geq 0$ such that $S_{\infty}=\{w\in \C\mid |w-\xi|\leq \lambda, \im w\geq 0\}$. 
Suppose $\lambda >0$. Then $\xi+\sqrt{-1} \lambda \in \partial S_{\infty}$ is in $\mathfrak h$. Since $\mathcal F$ is a fundamental domain, $\{\gamma \in \Gamma \mid \xi+\sqrt{-1} \lambda \in \gamma \mathcal F \}$ is a non-empty finite set. 
Thus we set $\{\delta_1, \dots, \delta_m\}:=\{\gamma \in \Gamma \mid \xi+\sqrt{-1} \lambda \in \gamma \mathcal F \}$, $V:=\bigcup_{j=1}^m\delta_j\mathcal F$, and put $U:=V^{\circ}$, where $V^{\circ}$ is the interior of $V$.
Then $U$ is an open neighborhood of $\xi+\sqrt{-1} \lambda$, and thus there exists $N >0$ such that $U$ intersects with $c_{l_n}$ for all $n\geq N$. Therefore, there exists $j \in \{1, \dots, m\}$ such that $\delta_j\mathcal F$ intersects with $c_{l_n}=B_{l_n}c_0$ for infinitely many $n \geq N$. 
Now suppose $\delta_j\mathcal F$ intersects with $c_{l_n}=B_{l_n}c_0$. Then by Lemma \ref{pf lem1} (2), there exists $k_n\in \Z$ such that $\gamma_0^{-k_n}B_{l_n}^{-1}\delta_j \mathcal F$ intersects with $\bsr{\tau_3, \gamma_0\tau_3}$. Since $\Gamma$ acts properly discontinuously on $\mathfrak h$, there exist $ n_1 > n_2 \geq N$ such that $\gamma_0^{-k_{n_1}}B_{l_{n_1}}^{-1}\delta_j\mathcal F=\gamma_0^{-k_{n_2}}B_{l_{n_2}}^{-1}\delta_j \mathcal F$. Thus $B_{n_1}=\pm B_{n_2}\gamma_0^{k_{n_2}-k_{n_1}}$, and hence $c_{n_1}=c_{n_2}$, which is a contradiction because $c_k$ are distinct by Lemma \ref{pf lem3.5} (1). Therefore $\lambda=0$ and we obtain (\ref{pf claim}) by Lemma \ref{pf lem0}.
\end{proof}

\subsection{Periodicity}\label{subsection periodicity}
Now we prove the Lagrange type periodicity theorem. 
We first prove the periodicity coming from the closed geodesics on $\Delta(2,3,7)\bs \mathfrak h$, and then prove a refined ``$\beta$-free'' version by using the convergence of the geodesic continued fraction. 
As we have remarked in Section \ref{intro}, the following Theorem \ref{lagrange closed geod} (1) follows directly from Proposition \ref{geodesic lagrange} and the well-known property of the geodesic continued fractions. However we have to discuss carefully in order to prove (2).

\begin{thm}[Lagrange's theorem for $\Delta(2,3,7)\bs \mathfrak h$]\label{lagrange closed geod}~
\begin{enumerate}[$(1)$]
\item Let $\alpha, \beta \in \R\cup \{\infty\}$ such that $\alpha\neq \beta$, and let $\varpi=\varpi_{\beta \to \alpha}$ be the oriented geodesic on $\mathfrak h$ joining $\beta$ to $\alpha$. Let 
\begin{align}
B_0^{-1}\varpi= \llbracket i_1,i_2,i_3, \cdots \rrbracket_{\Delta(2,3,7)}
\end{align}
be the geodesic continued fraction expansion of $\varpi$ with respect to $\Delta(2,3,7)$. For $k \geq 1$, set $A_k:=g_7^{i_k}g_2 \in \Delta(2,3,7)$, $B_k:=B_0A_1A_2\cdots A_k \in \Delta(2,3,7)$, and $\varpi_k:=B_k^{-1}\varpi$. 
Then the following conditions are equivalent.
\begin{enumerate}[{\rm (i)}]
\item The two endpoints $\alpha$ and $\beta$ are of the following form:
\begin{align}
\begin{cases}
\alpha &=\frac{1}{2w}(z-\bar z \pm \sqrt{D_{z,w}}) \\
\beta &= \frac{1}{2w}(z-\bar z \mp \sqrt{D_{z,w}})
\end{cases}
\end{align}
for some $z,w\in L$ such that $D_{z,w}:=(z-\bar z)^2+4\eta w\bar w > 0$. 
Here if $w=0$, we assume $(\alpha, \beta)=(0,\infty)$ or $(\infty,0)$.
\item There exists $l_0\geq 1$ such that $\varpi_{l_0}=\varpi_0$. (In particular the geodesic continued fraction expansion becomes periodic, i.e., $i_{k+l_0}=i_k$ for all $k\geq 1$.)
\end{enumerate}
\item Suppose that the above conditions are satisfied for $z,w \in L$ and $l_0 \geq 1$. 
Assume that $l_0$ is the minimal element such that the condition {\rm (ii)} holds.  
Put $\gamma_0:= B_{l_0}B_0^{-1}=B_0A_1\cdots A_{l_0}B_0^{-1}$. Then we have $\gamma_0 \in \Gamma_{\varpi} = \mathcal O_{z,w}^1$ (cf. Lemma \ref{lem ass field} (5)), and $\gamma_0$ gives the fundamental unit of $\mathcal O_{z,w}^1$. Equivalently, $\rho_{\alpha}(\gamma_0) \in F(\sqrt{D_{z,w}})$ gives the fundamental unit of $U_{\mathcal O,z,w/F}$.
\end{enumerate}
\end{thm}

\begin{proof}
(1) The implication (ii) $\Rightarrow$ (i) is clear from the implication (ii) $\Rightarrow$ (iii) of Proposition \ref{geodesic lagrange}. Indeed {\rm (ii)} implies $B_{l_0}B_0^{-1}\varpi=\varpi$, and we have $B_0 \mathcal D \neq B_{l_0}\mathcal D$ by Proposition \ref{prop GCF} (2) {\rm (iii)}. Therefore $B_{l_0}B_0^{-1} \neq \pm 1$ is a hyperbolic element in $\Gamma_{\varpi}$. 
We prove the implication (i) $\Rightarrow$ (ii). 
Since the conditions {\rm (i)} and {\rm (ii)} are preserved by replacing $\varpi$ with $\gamma \varpi$ for $\gamma \in \Delta(2,3,7)$ by Proposition \ref{prop GCF} (1), we may assume $\varpi$ is reduced (i.e., $B_0=1$) and $\varpi \cap \mathcal D^{\circ} \neq \emptyset$ by Lemma \ref{lem reduced} (\ref{reducibility2}). 
We take $R \in \varpi \cap \mathcal D^{\circ}$, and define $P_k, Q_k \in \varpi$ ($k\geq 0$) so that $\varpi \cap \mathcal D=\overrightarrow{P_kQ_k}$. 
By Proposition \ref{geodesic lagrange}, there exists a hyperbolic element in $\Gamma_{\varpi}$. Let  $\gamma_0 \in \Gamma_{\varpi}$ be any hyperbolic element. By replacing $\gamma_0$ by $\gamma_0^{-1}$ if necessary we assume $\alpha$ is the attracting point of $\gamma_0$. Then we have 
$R \leq_{\varpi} \gamma_0 R$. Furthermore, we have 
\begin{align}
P_0 <_{\varpi} R <_{\varpi} Q_0=P_1 <_{\varpi}  \gamma_0R. 
\end{align}
Indeed, if $\gamma_0 R \in \mathcal D$, then by Lemma \ref{lem reduced} (\ref{torsion}), we obtain $\gamma_0=\pm 1$, which is a contradiction. 
On the other hand, by Proposition \ref{prop GCF} and Theorem \ref{thm convergence}, we have $P_{k+1}=Q_k$ and $\lim_{k\to \infty}P_k=\lim_{k\to \infty}Q_k=\alpha$. 
Therefore,  there exists $l_0\geq 1$ such that $\gamma_0 R\in \bsr{P_{l_0},Q_{l_0}}$. 
Then we have $R \in \varpi\cap\mathcal D^{\circ}$ and $B_{l_0}^{-1}\gamma_0 R \in B_{l_0}^{-1}\gamma_0\varpi \cap \mathcal D$ and $\varpi$ and $B_{l_0}^{-1}\gamma_0\varpi=B_{l_0}^{-1}\varpi$ are both reduced. 
Therefore by Lemma \ref{lem reduced} (\ref{torsion}), we obtain $B_{l_0}^{-1}\gamma_0=\pm 1$, and hence $\varpi_{l_0}=\gamma_0 \varpi=\varpi_0$.

(2) We may assume $\varpi$ is reduced and $B_0=1$. 
Suppose $l_0$ is minimal. 
By the above argument, we see $B_{l_0} \in \Gamma_{\varpi} = \mathcal O_{z,w}^1$. On the other hand, let $\gamma_0$ be the hyperbolic element which generates of $\Gamma_{\varpi}$ up to $\pm 1$, and assume that $\alpha$ is the attracting point of $\gamma_0$. Then, again by the above argument, we obtain $\gamma_0=\pm B_{l_0'}$ for some $l_0' \geq 1$. Now by the periodicity (ii) and the minimality of $l_0$, we see $\gamma_0=B_{l_0'}=B_{m l_0}=B_{l_0}^m$ for some $m \geq 1$. Then, since $\gamma_0$ generates $\Gamma_{\varpi}$, we obtain $m=1$. Therefore, we get $l_0'=l_0$, and hence $B_{l_0}=\pm \gamma_0$ becomes the fundamental unit. This completes the proof.
\end{proof}

\paragraph{The $\beta$-free version}

In order to discuss the refined version of the above theorem, we first prepare some lemmas. 
For $\alpha \in \R \cup \{\infty\}$, we denote by $\mathcal G_{\alpha}$ the set of oriented geodesics on $\mathfrak h$ which goes towards $\alpha$. We naturally identify $\mathcal G_{\alpha}$ with $\P^1(\R)\!-\!\{\alpha\}=\R \cup\{\infty\}\!-\!\{\alpha\}$ as follows: 
\begin{align}
\P^1(\R)\!-\!\{\alpha\} \overset{\sim}{\longrightarrow} \mathcal G_{\alpha}; \beta \mapsto \varpi_{\beta \to \alpha}.
\end{align}
We equip $\mathcal G_{\alpha}$ with the natural topology of $\P^1(\R)$ via this identification. 
Then for $w\in \mathfrak h$, we denote by $p_{\alpha}(w) \in \mathcal G_{\alpha}$ the unique oriented geodesic on $\mathfrak h$ which passes through $w$ and goes to $\alpha$. This defines a map
\begin{align}
p_{\alpha}: \mathfrak h \rightarrow \mathcal G_{\alpha}
\end{align}
which is clearly continuous open map since $p_{\alpha}$ is a fiber bundle with fibers $\varpi_{\beta \to \alpha}$.

Now let $\varpi=\varpi_{\beta \to \alpha}$ be an oriented geodesic which satisfies the equivalent conditions of Proposition \ref{geodesic lagrange} and Theorem \ref{lagrange closed geod}. 
Furthermore we assume that $\varpi$ is {reduced} and {$\varpi \cap \mathcal D^{\circ} \neq \emptyset$}. 
Take $R \in \varpi \cap \mathcal D^{\circ}$ and a hyperbolic element $\gamma_0 \in \Gamma_{\varpi}$ which generates $\Gamma_{\varpi}$ up to $\pm 1$. We assume that $\alpha$ is the attracting point and $\beta$ is the repelling point of $\gamma_0$. Thus we have
\begin{align}
\varpi= \coprod_{n\in \Z}\gamma_0^n\bsr{R, \gamma_0R}.
\end{align}

We regard $\varpi$ as an element in $\mathcal G_{\alpha}$. Then $\mathcal G_{\alpha}\!-\!\{\varpi\}$ has two connected components. We denote by 
\begin{align}
\mathcal G_{\alpha,\beta,+} &:= \{\varpi_{\beta' \to \alpha} \mid \beta' \in (\beta, \alpha) \},    \\
\mathcal G_{\alpha,\beta,-} &:= \{\varpi_{\beta' \to \alpha} \mid \beta' \in (\alpha, \beta) \},
\end{align}
those two components. Here we use the notation (\ref{extended interval}). 
Let $S_0=\{w \in \C \mid |w| \leq \sqrt{\eta},\  \im z\geq 0\}$ be as before.

\begin{lem}\label{refine lem0}
Let $U \subset \mathcal G_{\alpha}$ be any connected open neighborhood of $\varpi$, then we have $\gamma_0^{-1}U \subset U$. In particular for all $\varpi' \in U\!-\!\{\varpi\}$, we have $\gamma_0^{-1}\varpi' \in  U\!-\!\{\varpi\}$.
\end{lem}
\begin{proof}
This follows from the assumption that $\gamma_0$ is a hyperbolic element with attracting fixed point $\alpha$ and repelling fixed point $\beta$.
\end{proof}

\begin{lem}\label{refine lem1}
There exists a connected open neighborhood $U \subset \mathcal G_{\alpha}$ of $\varpi$ such that for any $\varpi' \in U\!-\!\{\varpi\}$ we have $\varpi' \cap \mathcal D^{\circ} \neq \emptyset$ and 
\begin{align}
\varpi' \cap S_0 \cap \Delta(2,3,7) \tau_3 = \emptyset,
\end{align}
i.e., for every $\varpi' \in U\!-\!\{\varpi\}$, $\varpi'$ intersects with $\mathcal D^{\circ}$ and does not contain $\Delta(2,3,7)$-translations of $\tau_3$ inside $S_0$.
\end{lem}
\begin{proof}
Since $\mathcal D$, $\gamma_0 \mathcal D$, and $\bss{R,\gamma_0R}$ are all compact, there exists a geodesically convex open set $V \subset \mathfrak h$ containing $\bss{R,\gamma_0R}\cup \mathcal D \cup \gamma_0 \mathcal D$ such that its closure $\overline V$ (in $\mathfrak h$) is compact. Then, since the action of $\Delta(2,3,7)$ on $\mathfrak h$ is properly discontinuous, we see $V \cap \Delta(2,3,7)\tau_3$ is a finite set. Thus we take $w_1,\dots , w_r \in V$ such that
\begin{align}
\{w_1,\dots , w_r\}= V \cap \Delta(2,3,7)\tau_3.
\end{align}
Then $p_{\alpha}(V) \subset \mathcal G_{\alpha}$ becomes an open neighborhood of $\varpi$, and hence
\begin{align}
U_1=p_{\alpha}(V)\!-\!\{p_{\alpha}(w_1), \dots, p_{\alpha}(w_r)\} \cup \{\varpi\}
\end{align}
is also an open neighborhood of $\varpi$. 
Now since $\varpi$ enters $\mathcal D$ from $\mathbf e_0$ and $V$ is an open set containing $\mathcal D$, there exists $Q \in \varpi $ such that $Q \in V\!-\!S_0$. Thus we take an open neighborhood $W \subset V\!-\!S_0$ of $Q$ in $V\!-\!S_0$. 
Let $V_0 \subset \mathcal D$ be an open neighborhood of $R$ in $\mathcal D$. 
We define an open subset $U$ of $\mathcal G_{\alpha}$ to be the connected component of 
\begin{align}
U_1\cap p_{\alpha}(V_0) \cap p_{\alpha}(\gamma_0 V_0) \cap p_{\alpha}(W) \ni \varpi
\end{align}
containing $\varpi$. At this stage we know the following: for any $\varpi' \in U\!-\!\{\varpi\}$, we have 
\begin{align}\label{refine eqn V}
\varpi' \cap V \cap \Delta(2,3,7) \tau_3 =\emptyset.
\end{align}
Therefore it remains to extend the $\Delta(2,3,7)\tau_3$-free region from $V$ to $S_0$. We prove the following:
\begin{claim}\label{refine claim1}
Let $\varpi' \in U\!-\!\{\varpi\}$, then we have 
\begin{align}
\varpi' \cap S_0 \subset \bigcup_{n \geq 0}\gamma_0^n V.
\end{align}
\end{claim}
Now, by Lemma \ref{refine lem0}, we have $\gamma_0^{-n} \varpi' \in U\!-\!\{\varpi\}$ for all $n \geq 0$, and hence $\gamma_0^{-n}\varpi' \in p_{\alpha}(V_0) \cap p_{\alpha}(\gamma_0V_0)$ for all $n \geq 0$. Therefore we see $\gamma_0^{-n}\varpi'$ intersects with $V_0$ and $\gamma_0V_0$ for all $n\geq 0$, or equivalently, $\varpi'$ intersects with $\gamma_0^nV_0$ and $\gamma_0^{n+1}V_0$ for all $n\geq 0$.
On the other hand, by the definition, $\gamma_0^nV$ is a geodesically convex open set which contains $\gamma_0^nV_0$ and $\gamma_0^{n+1}V_0$. Therefore the geodesic segment of $\varpi'$ between $\varpi' \cap \gamma_0^nV_0$ and $\varpi' \cap \gamma_0^{n+1}V_0$ is contained in $\gamma_0^nV$. 
Now, since $\alpha$ is the attracting point of $\gamma_0$, the points in $\gamma_0^nV$ converges to $\alpha$ uniformly as $n\to \infty$. 
Therefore the geodesic segment of $\varpi'$ from $\varpi' \cap V$ to $\alpha$ is contained in $\bigcup_{n \geq 0}\gamma_0^n V$.
On the other hand since $W \subset V\!-\!S_0$ and $\varpi' \in p_{\alpha}(W)$, we see that $\varpi' \cap S_0$ is contained in the geodesic segment of $\varpi'$ from $\varpi' \cap V$ to $\alpha$. This proves the claim.

Now let $\varpi' \in U\!-\!\{\varpi\}$ and let $P \in \varpi'\cap S_0$. Then by Claim \ref{refine claim1}, there exists $n \geq 0$ such that $\gamma_0^{-n}P \in \gamma_0^{-n}\varpi' \cap V$. 
On the other hand, we have $\gamma_0^{-n} \varpi' \in U\!-\!\{\varpi\}$ by Lemma \ref{refine lem0}. Thus we get $\gamma_0^{-n} \varpi' \cap V \cap \Delta(2,3,7)\tau_3 =\emptyset$ by (\ref{refine eqn V}). Therefore $P \notin \Delta(2,3,7)\tau_3$. This proves the lemma.
\end{proof}

Let $U \subset \mathcal G_{\alpha}$ be a connected open neighborhood of $\varpi$ satisfying the condition in Lemma \ref{refine lem1}. Then $U \!-\!\{\varpi\}$ consists of two connected components 
\begin{align}
U_+ &=U\cap \mathcal G_{\alpha,\beta, +} \label{U+} \\
U_- &=U\cap \mathcal G_{\alpha,\beta, -}. \label{U-}
\end{align}

\begin{lem}\label{refine lem2}
Let $\varpi^{(1)}, \varpi^{(2)} \in U_+$ (or $\varpi^{(1)}, \varpi^{(2)} \in U_-$) be geodesics in the same connected component of $U \!-\!\{\varpi\}$. By the assumption, we can take $R_1 \in \varpi^{(1)}\cap \mathcal D^{\circ}$ and $R_2 \in \varpi^{(2)}\cap \mathcal D^{\circ}$. Let 
\begin{align}
\varpi^{(1)} = \llbracket B_{1,0}; i_1, i_2, \dots \rrbracket_{\Delta(2,3,7)}, \\
\varpi^{(2)} = \llbracket B_{2,0}; j_1, j_2, \dots \rrbracket_{\Delta(2,3,7)}
\end{align}
be the geodesic continued fraction expansions of $\varpi^{(1)}$ and $\varpi^{(2)}$ such that $B_{1,0}^{-1}R_1 \in \mathcal D^{\circ}$ and $B_{2,0}^{-1}R_2 \in \mathcal D^{\circ}$. Then we have $B_{1,0}=\pm B_{2,0}$ and $i_k=j_k$ for all $k \geq 1$. 
\end{lem}
\begin{proof} 
We only prove the case $\varpi^{(1)}, \varpi^{(2)} \in U_+$ since the case $\varpi^{(1)}, \varpi^{(2)} \in U_-$ can be proved similarly. 
By Lemma \ref{lem reduced} (\ref{torsion}) and the definition of the algorithm, it suffices to show the following: 
\begin{claim}\label{refine claim2}
Suppose $\gamma \in \Delta(2,3,7)$ satisfies $\gamma \mathcal D \cap  \varpi^{(1)} \neq \emptyset$ and $\gamma \mathcal D \cap  \varpi^{(2)} \neq \emptyset$ and $\gamma \mathcal D \subset S_0$. Then both $\gamma^{-1} \varpi^{(1)}$ and $\gamma^{-1} \varpi^{(2)}$ intersect with $\mathcal D^{\circ}$ and enter (resp. leave) $\mathcal D$ from the same edge, say $\mathbf e_i$ (resp. $\mathbf e_j'$). 
\end{claim}
By Lemma \ref{refine lem1}, $\gamma \mathcal D \cap \varpi^{(1)}$ and $\gamma \mathcal D \cap \varpi^{(2)}$ do not contain vertices of $\gamma \mathcal D$. Therefore $\gamma^{-1} \varpi^{(1)}$ and $\gamma^{-1} \varpi^{(2)}$ must intersect with $\mathcal D^{\circ}$. Now let $I \subset U_+$ be the closed interval between $\varpi^{(1)}$ and $\varpi^{(2)}$ contained in $U_+$ (via the identification $\mathcal G_{\alpha} \simeq \P^1(\R) \!-\!\{\alpha\}$). Then because $U$ satisfies the condition in Lemma \ref{refine lem1}, we have
\begin{align}
p_{\alpha}^{-1}(I) \cap S_0 \cap \Delta(2,3,7)\tau_3 = \emptyset. \label{refine eqn claim2}
\end{align}
Suppose $\gamma^{-1} \varpi^{(1)}$ enters (resp. leaves) $\mathcal D$ from $\mathbf e_i$ (resp. $\mathbf e_j'$). Now, since the endpoints $\gamma g_7^i\tau_3$ and $\gamma g_7^i\tau_3'$ (resp. $\gamma g_7^j\tau_3$ and $\gamma g_7^j\tau_3'$) of $\gamma \mathbf e_i$ (resp. $\gamma \mathbf e_j'$) are contained in $S_0 \cap \Delta(2,3,7)\tau_3$, they must not be contained in $p_{\alpha}^{-1}(I)$ by (\ref{refine eqn claim2}). 
On the other hand, the geodesic segment $\gamma \mathbf e_i$ (resp. $\gamma \mathbf e_j'$) can intersect with $\varpi^{(1)}$ at most once. 
Therefore, $\gamma \mathbf e_i$ (resp. $\gamma \mathbf e_j'$) must also intersect with $\varpi^{(2)}$. 
Finally, if $\gamma^{-1}\varpi^{(2)}$ leaves (resp. enters) $\mathcal D$ from $\mathbf e_i$ (resp. $\mathbf e_j'$), this contradicts to the assumption that $\varpi^{(2)}$ is an oriented geodesic which goes to $\alpha$. 
This proves the claim and hence the lemma.
\end{proof}

Now we present the refined version of the above theorem in which we can get rid of the condition on $\beta$.

\begin{thm}[Lagrange's theorem for $\Delta(2,3,7)\bs \mathfrak h$, $\beta$-free version]\label{thm lagrange}~
\begin{enumerate}[$(1)$]
\item Let $\alpha, \beta \in \R\cup \{\infty\}$ such that $\alpha\neq \beta$, and let $\varpi=\varpi_{\beta \to \alpha}$ be the oriented geodesic on $\mathfrak h$ joining $\beta$ to $\alpha$. Let 
\begin{align}
B_0^{-1}\varpi= \llbracket i_1,i_2,i_3, \cdots \rrbracket_{\Delta(2,3,7)} \label{lagrange thm GCF exp}
\end{align}
be the geodesic continued fraction expansion of $\varpi$ with respect to $\Delta(2,3,7)$. For $k \geq 1$, set $A_k:=g_7^{i_k}g_2 \in \Delta(2,3,7)$ and $B_k:=B_0A_1A_2\cdots A_k \in \Delta(2,3,7)$. 
Then the following conditions are equivalent.
\begin{enumerate}[{\rm (i)}]
\item The endpoint $\alpha$ is of the following form:
\begin{align}
\alpha =\frac{1}{2w}(z-\bar z \pm \sqrt{D_{z,w}}) \label{eqn alpha}
\end{align}
for some $z,w\in L$ such that $D_{z,w}:=(z-\bar z)^2+4\eta w\bar w > 0$. 
Here when $w=0$, we assume $\alpha=\infty$ (resp. $0$) if the sign in $\alpha$ is $+$ (resp. $-$).
\item There exist $l_0 \geq 1$ and $k_0 \geq 0$ such that $B_{k+l_0}^{-1}\alpha=B_k^{-1}\alpha$ for all $k \geq k_0$.
\item There exist $l_0 \geq 1$ and $k_0 \geq 0$ such that $i_{k+l_0}=i_k$ for all $k > k_0$, i.e., the geodesic continued fraction expansion becomes periodic.
\end{enumerate}
\item Suppose that the above conditions are satisfied for $z,w \in L$ and $l_0, k_0\geq 1$. 
Assume that $l_0$ is the minimal element such that the condition {\rm (iii)} holds.  
Put $\gamma_0:= B_{k_0+l_0}B_{k_0}^{-1}=B_{k_0}A_{k_0+1}\cdots A_{k_0+l_0}B_{k_0}^{-1}$. Then we have $\gamma_0 \in \mathcal O_{z,w}^1$, and $\gamma_0$ gives the fundamental unit of $\mathcal O_{z,w}^1$. Equivalently, $\rho_{\alpha}(\gamma_0) \in F(\sqrt{D_{z,w}})$ gives the fundamental unit of $U_{\mathcal O,z,w/F}$.
\end{enumerate}
\end{thm}
\begin{proof}
(1) The implication (ii) $\Rightarrow$ (i) is clear. The implication (iii) $\Rightarrow$ (ii) follows directly from the convergence of geodesic continued fraction (Corollary \ref{cor convergence}). We prove (i) $\Rightarrow$ (iii). 
Let $\beta':=\frac{z-\bar z}{w} -\alpha$, and let $\varpi':=\varpi_{\beta'\to \alpha}$ be the oriented geodesic joining $\beta'$ to $\alpha$. 
Here if $w=0$ and $\alpha=0$ (resp. $\alpha=\infty$), we assume $\beta'=\infty$ (resp. $\beta'=0$).
We use this auxiliary geodesic $\varpi'$ about which we have already studied in Theorem \ref{lagrange closed geod}. In particular the case where $\beta=\beta'$ is already proved in Theorem \ref{lagrange closed geod}. Therefore we assume $\beta \neq \beta'$. 
By Proposition \ref{geodesic lagrange}, there exists a hyperbolic element in $\Gamma_{\varpi'}$. Let  $\delta_0 \in \Gamma_{\varpi'}$ be any hyperbolic element.
We assume that $\alpha$ is the attracting point and $\beta'$ is the repelling point of $\delta_0$. The key fact to prove this theorem is $\lim_{n\to \infty}\delta_0^{-n}\beta=\beta'$. 

By Lemma \ref{lem reduced} (\ref{reducibility2}) and Proposition \ref{prop GCF} (1), we may assume $\varpi'$ is reduced and $\varpi' \cap \mathcal D^{\circ}\neq \emptyset$. Take $R' \in \varpi' \cap \mathcal D^{\circ}$. 
Then we can use Lemma \ref{refine lem1} to take a connected open neighborhood $U \subset \mathcal G_{\alpha}$ of $\varpi'$ 
such that for any $\varpi'' \in U\!-\!\{\varpi'\}$, we have $\varpi'' \cap \mathcal D^{\circ} \neq \emptyset$, and 
\begin{align}
\varpi'' \cap S_0 \cap \Delta(2,3,7) \tau_3 = \emptyset.
\end{align}
We denote by $U_+,U_-$ the connected components of $U\!-\!\{\varpi'\}$ as in (\ref{U+}), (\ref{U-}). 
Since we have assumed that $\beta \neq \beta'$, we have $\varpi \in \mathcal G_{\alpha} \!-\!\{\varpi'\}= \mathcal G_{\alpha,\beta',+} \cup \mathcal G_{\alpha,\beta',-}$. 
Suppose $\varpi \in \mathcal G_{\alpha,\beta',+}$ (resp. $\mathcal G_{\alpha,\beta',-}$). 
Then, since we have $\lim_{n\to \infty}\delta_0^{-n}\beta=\beta'$ and $\delta_0 \mathcal G_{\alpha,\beta',+}=\mathcal G_{\alpha,\beta',+}$ (resp. $\delta_0 \mathcal G_{\alpha,\beta',-}=\mathcal G_{\alpha,\beta',-}$), there exists $N_1 \geq 0$ such that $\delta_0^{-n} \varpi \in U_+$ (resp. $U_-$) for all $n \geq N_1$.
Take $R_n \in \delta_0^{-n} \varpi \cap \mathcal D^{\circ}$ for each $n\geq N_1$. Since $\alpha$ is the attracting point of $\delta_0$, we have $\lim_{n \to \infty} \delta_0^nR_n=\alpha$. 

Now, as in the proof of Theorem \ref{lagrange closed geod}, we define $P_k, Q_k \in \varpi$ ($k\geq 0$) so that $\varpi \cap B_k\mathcal D=\overrightarrow{P_kQ_k}$. 
Let us fix a constant $M \in \Z_{\geq 0}$ arbitrarily. (We use this later.)
By Proposition \ref{prop GCF} and Theorem \ref{thm convergence}, we have $P_{k+1}=Q_k$ and $\lim_{k\to \infty}P_k=\lim_{k\to \infty}Q_k=\alpha$. Therefore, there exist $N_2 \geq N_1$, $k_0 \geq M$, and $l_0 \geq 1$ such that 
\begin{align}
\delta_0^{N_2} R_{N_2} & \in \bsr{P_{k_0}, Q_{k_0}} \subset \varpi,  \label{RPQ1}\\ 
\delta_0^{N_2+1} R_{N_2+1} & \in \bsr{P_{k_0+l_0}, Q_{k_0+l_0}} \subset \varpi. \label{RPQ2}
\end{align}
Let 
\begin{align}
C_0^{-1}(\delta_0^{-N_2}\varpi) &=\llbracket j_1, j_2, \dots \rrbracket_{\Delta(2,3,7)} \label{eqn350}\\
(C_0')^{-1}(\delta_0^{-N_2-1}\varpi) &=\llbracket j_1', j_2', \dots \rrbracket_{\Delta(2,3,7)} \label{eqn351}
\end{align}
be the geodesic continued fraction expansions of $\delta_0^{-N_2}\varpi$ and $\delta_0^{-N_2-1}\varpi$ such that 
\begin{align}
C_0^{-1}R_{N_2} \in \mathcal D^{\circ} \quad \text{ and } \quad (C_0')^{-1}R_{N_2+1} \in \mathcal D^{\circ}. \label{eqn3522}
\end{align} 
Then, by Lemma \ref{refine lem2}, we have $C_0=\pm C_0'$ and $j_k=j_k'$ for all $k\geq 1$. On the other hand, from the geodesic continued fraction expansion (\ref{lagrange thm GCF exp}), we also obtain the following geodesic continued fraction expansions: 
\begin{align}
B_{k_0}^{-1}\varpi &=\llbracket i_{k_0+1}, i_{k_0+2}, \dots \rrbracket_{\Delta(2,3,7)}, \label{eqn352}\\
B_{k_0+l_0}^{-1}\varpi &=\llbracket i_{k_0+l_0+1}, i_{k_0+l_0+2}, \dots \rrbracket_{\Delta(2,3,7)}, \label{eqn353}
\end{align}

Now we apply Lemma \ref{lem reduced} (\ref{torsion}) to 
\begin{itemize}
\item the reduced oriented geodesic: $C_0^{-1}(\delta_0^{-N_2}\varpi)$ (resp. $(C_0')^{-1}(\delta_0^{-N_2-1}\varpi)$), 
\item point: $z=C_0^{-1}R_{N_2} \in \mathcal D^{\circ}$, (resp. $z=(C_0')^{-1}R_{N_2+1} \in \mathcal D^{\circ}$),  
\item $\gamma= B_{k_0}^{-1}\delta_0^{N_2}C_0 \in \Delta(2,3,7)$, (resp. $ \gamma= B_{k_0+l_0}^{-1}\delta_0^{N_2+1}C_0' \in \Delta(2,3,7)$).
\end{itemize}
Because we have
\begin{itemize}
\item $\gamma C_0^{-1}(\delta_0^{-N_2}\varpi)=B_{k_0}^{-1}\varpi $ (resp. $\gamma (C_0')^{-1}(\delta_0^{-N_2-1}\varpi)=B_{k_0+l_0}^{-1}\varpi $) is reduced by the definition of the geodesic continued fraction expansion, 
\item $\gamma z= B_{k_0}^{-1}\delta_0^{N_2}R_{N_2} \in \mathcal D$ (resp. $\gamma z= B_{k_0+l_0}^{-1}\delta_0^{N_2+1}R_{N_2+1} \in \mathcal D$) by  (\ref{RPQ1}) (resp. (\ref{RPQ2})),
\end{itemize}
%
we obtain
\begin{align}
B_{k_0}^{-1}\delta_0^{N_2}C_0 =\pm 1 \quad \text{ and } \quad B_{k_0+l_0}^{-1}\delta_0^{N_2+1}C_0' =\pm 1. \label{eqn356}
\end{align} 
Then it follows that both (\ref{eqn350})  and (\ref{eqn352}) give the geodesic continued fraction expansion of $C_0^{-1}\delta_0^{-N_2}\varpi = B_{k_0}^{-1}\varpi$, and both (\ref{eqn351})  and (\ref{eqn353}) give the geodesic continued fraction expansion of $(C_0')^{-1} \delta_0^{-N_2-1}\varpi = B_{k_0+l_0}^{-1}\varpi$. 
Therefore, by using Proposition \ref{prop GCF} (1), we finally obtain
\begin{align}
i_{k_0+k}=j_k=j'_k=i_{k_0+l_0+k} \label{eqn357}
\end{align}
for all $k\geq 1$. 
Thus we obtain (iii). 

(2) We may assume $\beta \neq \beta'$, $\varpi'$ is reduced and $\varpi \cap \mathcal D^{\circ} \neq \emptyset$, where $\beta'$ and $\varpi'$ are the same objects as in the proof of the implication (i) $\Rightarrow$ (iii). 
Suppose that the equivalent conditions (i) $\sim$ (iii) are satisfied. More precisely, suppose that the condition (iii) is satisfied for $k_0,l_0$, and that $l_0$ is the minimal element such that the condition {\rm (iii)} holds. Then clearly the condition (ii) is also satisfied for the same $k_0$ and $l_0$. 
By the condition (ii), we have $B_{k_0+l_0}B_{k_0}^{-1}\alpha =\alpha$. Therefore, by Lemma \ref{lem ass field} (5), we obtain $\gamma_0:=B_{k_0+l_0}B_{k_0}^{-1} \in K_{z,w} \cap \mathcal O^1= \mathcal O_{z,w}^1 =\Gamma_{\varpi'}$. 
On the other hand, let $\delta_0$ be the hyperbolic element which generates $\Gamma_{\varpi'}=\mathcal O_{z,w}^1$ up to $\pm 1$, and assume that $\alpha$ is the attracting point of $\delta_0$. Then, by the above argument, especially from (\ref{eqn356}) and (\ref{eqn357}), 
there exist $k_0' \geq k_0$ and $l_0'\geq 1$ such that $i_{k+l_0'}=i_{k}$ for all $k >k_0'$ and $\delta_0=\pm B_{k_0'+l_0'}B_{k_0'}^{-1}$. (Here we choose $M=k_0$ in the above argument.)
Now, by the periodicity (iii), we have $A_{k+l_0}=A_{k}$ for all $k \geq k_0$. Therefore we have $\gamma_0=B_{k_0+l_0}B_{k_0}^{-1}=B_{k_0'+l_0}B_{k_0'}^{-1}$ by the definition of $B_k$. 
Thus, again by the periodicity (iii) and the minimality of $l_0$, we obtain 
\begin{align}
\delta_0=\pm B_{k_0'+l_0'}B_{k_0'}^{-1}=B_{k_0'+m l_0}B_{k_0'}^{-1}=(B_{k_0'+l_0}B_{k_0'}^{-1})^m=\gamma_0^m
\end{align}
for some $m \geq 1$. 
Then, since $\delta_0$ generates $\Gamma_{\varpi'}$, we obtain $m=1$. Therefore, we get $l_0'=l_0$, and hence $\gamma_0=B_{k_0+l_0}B_{k_0}^{-1} =\pm \delta_0$ becomes the fundamental unit. This completes the proof.
\end{proof}

\begin{rmk}\label{rmk beta free}
The results in \cite{series86} and \cite{abramskatok19} are interesting, and seem to be related to the ``$\beta$-free'' version of the periodicity. Although their results do not cover the case of the $(2,3,7)$-triangle group, they compare the geodesic continued fractions (Morse codings) and the ``boundary expansions'' of geodesics which essentially depends only on the one end point $\alpha$. It may be possible to prove (1) of Theorem \ref{thm lagrange} using the similar arguments to those in \cite{series86} and \cite{abramskatok19}. Our proof is different from their methods. 
\end{rmk}

\subsection{Some explicit constants}\label{constants}

Here we present some explicit computation of the elements such as $g_2, g_3, g_7$, $\bm a_i, \bm b_i, \bm c_i$ ($i \in \Z/7\Z \!-\!\{0\}$). 

Recall that $A$ is the quaternion algebra $\quat{\eta,\eta}{F}=F+Fi+Fj+Fk$ over $F=\Q(\eta)$, where $i,j,k$ satisfy $i^2=j^2=\eta$ and $ij=-ji=k$, and $L=\Q(\sqrt{\eta}) \subset \R$ is the fixed splitting field of $A$. 
We have embedded $A$ into $M_2(\R)$ as
\begin{align}
\iota: A 
\overset{\sim}{\rightarrow} \left\{ \left.
\begin{pmatrix}
z&\eta \bar{w}\\
w&\bar{z}
\end{pmatrix}\right|
z,w \in L
\right\} \subset M_2(\R); 
\begin{cases}
i &\mapsto 
\begin{pmatrix}
0&\eta \\
1&0\\
\end{pmatrix}\vspace{1mm}\\
j &\mapsto 
\begin{pmatrix}
\sqrt{\eta}&0\\
0&-\sqrt{\eta}\\
\end{pmatrix}
\end{cases}.
\end{align}
Then the elements $g_2, g_3, g_7 \in \Delta(2,3,7) \subset SL_2(\R)$ (cf. (\ref{g2}), (\ref{g3}), (\ref{g7})) are of the following form: 
\begin{align}
g_2&=
\begin{pmatrix}
0&-\sqrt{\eta}\\
1/\sqrt{\eta}&0\\
\end{pmatrix}, \label{g_2} \\
g_3&=\frac{1}{2}
\begin{pmatrix}
1+(\eta^2-2)\sqrt{\eta}&-(\eta^2+\eta-1)\sqrt{\eta}\\
(3-\eta^2)\sqrt{\eta}&1-(\eta^2-2)\sqrt{\eta}\\
\end{pmatrix}, \label{g_3} \\
g_7&=\frac{1}{2}
\begin{pmatrix}
\eta^2+\eta-1&\eta^2-1-\sqrt{\eta}\\
2-\eta^2+(\eta^2+\eta-2)\sqrt{\eta}&\eta^2+\eta-1\\
\end{pmatrix}. \label{g_7}
\end{align}
Put $\theta=\sqrt{\eta}$ for simplicity. Then we can compute $\bm a_i, \bm b_i, \bm c_i$ ($i \in \Z/7\Z\!-\!\{0\}$) as follows:
\begin{align}
&\left\{
\begin{alignedat}{4}
\bm a_1&=-\bm a_{-1}& &= \theta-\theta^2+\theta^4-\theta^5 \quad\quad\quad~& &=-\eta+\eta^2+(1-\eta^2)\sqrt{\eta} \quad\quad~  \\
\bm b_1&=\bm b_{-1}&  &= 4+8\theta^2-8\theta^4 & &= 4(1+2\eta -2\eta^2)    \\
\bm c_1&=-\bm c_{-1}& &= \theta+\theta^2-\theta^4-\theta^5 & &=\eta-\eta^2+(1-\eta^2)\sqrt{\eta} 
\end{alignedat}
\right. \label{eqn a1}\\
&\left\{
\begin{alignedat}{4}
\bm a_2&=-\bm a_{-2}& &= -2-3\theta+\theta^2+\theta^4+\theta^5~& &= -2+\eta+\eta^2+(-3+\eta^2)\sqrt{\eta}   \\
\bm b_2&=\bm b_{-2}& &= -8+4\theta^2 +4\theta^4& &=4(-2+\eta+\eta^2)    \\
\bm c_2&=-\bm c_{-2}& &= 2-3\theta-\theta^2-\theta^4+\theta^5& &= 2-\eta-\eta^2+(-3+\eta^2)\sqrt{\eta}
\end{alignedat}
\right. \label{eqn a2} \\
&\left\{
\begin{alignedat}{4}
\bm a_3&=-\bm a_{-3}& &= -\theta+\theta^2-2\theta^3+\theta^4-\theta^5 & &=\eta+\eta^2-(1+2\eta+\eta^2)\sqrt{\eta}   \\
\bm b_3&=\bm b_{-3}& &= 4+12\theta^2+4\theta^4 & &= 4(1+3\eta +\eta^2)    \\
\bm c_3&=-\bm c_{-3}& &= -\theta-\theta^2-2\theta^3-\theta^4-\theta^5 & &=-\eta-\eta^2-(1+2\eta+\eta^2)\sqrt{\eta}. 
\end{alignedat}
\right. \label{eqn a3}
\end{align}
We can also easily check the following properties of these constants. For $j \in \Z/7\Z\!-\!\{0\}$,
\begin{align}
\bm a_j &= g_7^j~ 0 = (g_7^j g_2) \infty  \in \mathcal O_L,  \label{ai geom}\\
-\bm c_j &=\overline{\bm a_j} = (g_7^j g_2)^{-1} \infty  \in  \mathcal O_L, \\
\bm b_j/4  &= \frac{\eta}{\left(2 \cos \left(\frac{j \pi}{7}\right)\right)^2} = \frac{2\cos\left(\frac{2\pi}{7}\right)}{\left(2 \cos \left(\frac{ j \pi}{7}\right)\right)^2} \in \mathcal O_F^{\times}, \\
|\bm a_1|&=|\bm a_{-1}|<|\bm a_2|=|\bm a_{-2}|<\sqrt{\eta} < |\bm a_3|=|\bm a_{-3}|. \label{ai value}
\end{align}

\subsection{Examples}\label{subsection examples}
Now we present some examples to illustrate our main theorems (Theorem \ref{lagrange closed geod} and Theorem \ref{thm lagrange}).
We describe the following items:
\begin{enumerate}[(I)]
\item Input data: $(\alpha, \beta, \varpi_{\beta \to \alpha})$, where $\alpha, \beta \in \R\cup\{\infty\}$ such that $\alpha \neq \beta$ and $\varpi_{\beta \to \alpha}$ is the oriented geodesic joining $\beta$ to $\alpha$ as before.
\item The resulting geodesic continued fraction expansion of $\varpi_{\beta \to \alpha}$.
\end{enumerate}
Suppose the geodesic continued fraction expansion of $\varpi_{\beta \to \alpha}$ becomes periodic (which is the case we are mainly interested in). 
As in the classical theory of continued fraction, we denote by 
 \begin{align} 
 \varpi_{\beta \to \alpha}= \llbracket B_0; i_1, \dots, i_{k_0}, \overline{i_{k_0+1} , \dots, i_{k_0+l_0}} \rrbracket_{\Delta(2,3,7)}, 
 \end{align}
the periodic geodesic continued fraction expansion with period $i_{k_0+1} , \dots, i_{k_0+l_0}$.
Then, furthermore we present
\begin{enumerate}
\item[(I)$'$] The data $z,w \in L$ for which $\alpha$ can be expressed as in Theorem \ref{thm lagrange} (i), and the associated quadratic extension $K_{z,w}\simeq F(\sqrt{D_{z,w}})$ over $F$.
\item[(III)]  The fundamental unit of $\mathcal O_{z,w}^1 \simeq U_{\mathcal O,z,w/F}$ obtained from the period of geodesic continued fraction expansion. Cf. Theorem \ref{lagrange closed geod} and Theorem \ref{thm lagrange}.
\end{enumerate}

We put $\theta:=\sqrt{\eta}$ as before. 
Moreover, for the geodesic continued fraction expansion (\ref{GCF exp algo}) of $\varpi$, we put $A_k:=g_7^{i_k}g_2$, $B_k:=B_0A_1\cdots A_k$ and $\varpi_k:=B_k^{-1}\varpi$. 
\begin{ex}\label{ex0}
\begin{enumerate}[(I)]
\item Input: $\alpha=0$, $\beta=\infty$, $\varpi:= \varpi_{\beta \to \alpha}$. 
\item[(I)$'$] The corresponding data: $z=\sqrt{\eta}$, $w=0$, $D_{z,w}=4\eta$. \\ 
The associated quadratic extension: $K_{z,w} \simeq F(\sqrt{\eta})=L$. 
\item The geodesic continued fraction expansion:
\begin{align}
\varpi = \llbracket 1; \overline{3, -2, 3} \rrbracket_{\Delta(2,3,7)}.
\end{align}
\item From the period, we obtain the following fundamental unit $\gamma_0$, i.e., a generator of $\mathcal O_{z,w}^1 = \Gamma_{\varpi_{\beta \to \alpha}}$ up to $\pm 1$: 
\begin{align}
\gamma_0 &:=B_3B_0^{-1}=(g_7^3g_2)(g_7^{-2}g_2)(g_7^3g_2)  \\
&= 
\begin{pmatrix}
-1-\theta-\theta^2 +\theta^3+\theta^5&0\\
0&-1+\theta-\theta^2 -\theta^3-\theta^5 \\
\end{pmatrix}\\
&=
\begin{pmatrix}
-1-\eta -(1-\eta-\eta^2) \sqrt{\eta}&0\\
0&-1-\eta +(1-\eta-\eta^2) \sqrt{\eta} \\
\end{pmatrix}.
\end{align}
Under the identification (\ref{ev map}), the fundamental unit $\varepsilon_0=\rho_{\alpha}(\gamma_0)$ of $U_{\mathcal O,z,w/F}$ can be written as 
\begin{align}
\varepsilon_0 = -1-\eta +(1-\eta-\eta^2) \sqrt{\eta} \in U_{\mathcal O,z,w/F}. 
\end{align}
\end{enumerate}
This example gives an example in which the traditional $k$-th convergent $\bm x^{trad}_k$ does not converges to $\alpha=0$. 
Indeed, by (\ref{conv trad2}), we see
\begin{align}
\bm x^{trad}_{3k-1}=((g_7^3g_2)(g_7^{-2}g_2)(g_7^3g_2))^k \infty =\infty \nrightarrow 0.
\end{align}
\end{ex}

\begin{ex}\label{ex1}
\begin{enumerate}[(I)]
\item Input: $\alpha=(1-\eta^2)\sqrt{\eta} +\sqrt{1+3\eta-2\eta^2}$, $\beta=(1-\eta^2)\sqrt{\eta} -\sqrt{1+3\eta-2\eta^2}$, $\varpi:= \varpi_{\beta \to \alpha}$. 
\item[(I)$'$] The corresponding data: $z=(1-\eta^2)\sqrt{\eta}$, $w=1$, $D_{z,w}=4(1+3\eta-2\eta^2)$. \\ 
The associated quadratic extension: $K_{z,w} \simeq F(\sqrt{1+3\eta-2\eta^2})$. 
\item The geodesic continued fraction expansion:
\begin{align}
\varpi = \llbracket 1; \overline{1, -1} \rrbracket_{\Delta(2,3,7)}.
\end{align}
\item From the period, we obtain the following fundamental unit $\gamma_0$, i.e., a generator of $\mathcal O_{z,w}^1 = \Gamma_{\varpi_{\beta \to \alpha}}$ up to $\pm 1$: 
\begin{align}
\gamma_0 &:=B_2B_0^{-1}=(g_7g_2)(g_7^{-1}g_2) \\
&= \frac{1}{2}
\begin{pmatrix}
-1+2 \theta-\theta^2 - \theta^5&-1\\
2-\theta^2 -\theta^4&-1-2\theta-\theta^2 +\theta^5 \\
\end{pmatrix}\\
&= \frac{1}{2}
\begin{pmatrix}
 -(1+\eta)+ (2-\eta^2)\sqrt{\eta}&-1\\
2-\eta-\eta^2& -(1+\eta)- (2-\eta^2)\sqrt{\eta} \\
\end{pmatrix}.
\end{align}
Under the identification (\ref{ev map}), the fundamental unit $\varepsilon_0=\rho_{\alpha}(\gamma_0)$ of $U_{\mathcal O,z,w/F}$ can be written as 
\begin{align}
\varepsilon_0 
&=\frac{1}{2}((2-\eta-\eta^2)\alpha + -(1+\eta)- (2-\eta^2)\sqrt{\eta} )\\
&= -\frac{1}{2}\left(1+\eta +\sqrt{\eta^2+2\eta-3}\right) \in U_{\mathcal O,z,w/F}.
\end{align}
\end{enumerate}
By Corollary \ref{conv cor2}, the traditional $k$-th convergent of the following formal continued fraction converges to $(1-\eta^2)\sqrt{\eta} +\sqrt{1+3\eta-2\eta^2}$. 
\begin{align}
(1-\eta^2)\sqrt{\eta} +\sqrt{1+3\eta-2\eta^2} =\bm a_{1}-\cfrac{\bm b_{1}}{\bm c_{1}+\bm a_{-1}-\cfrac{\bm b_{-1}}{\bm c_{-1}+\bm a_{1}-\cfrac{\bm b_{1}}{\bm c_{1}+\bm a_{-1}-\cfrac{\bm b_{-1}}{\bm c_{-1}+\cdots}}}}.
\end{align}
Therefore, we can simplify this continued fraction using the formulas (\ref{eqn a1}), and obtain the continued fraction (\ref{eqn ex1}) presented in Section \ref{intro}.
\end{ex}

\begin{ex}\label{ex1.5} We give an example of ``$\beta$-free'' variant of Example \ref{ex1}
\begin{enumerate}[(I)]
\item Input: $\alpha=(1-\eta^2)\sqrt{\eta} +\sqrt{1+3\eta-2\eta^2}$, $\beta=-1$, $\varpi:= \varpi_{\beta \to \alpha}$. 
\item[(I)$'$] The corresponding data: $z=(1-\eta^2)\sqrt{\eta}$, $w=1$, $D_{z,w}=4(1+3\eta-2\eta^2)$. \\ 
The associated quadratic extension: $K_{z,w} \simeq F(\sqrt{1+3\eta-2\eta^2})$. 
\item The geodesic continued fraction expansion:
\begin{align}
\varpi = \llbracket g_7^{-1}; 3,2,\overline{-1, 1} \rrbracket_{\Delta(2,3,7)}.
\end{align}
\item From the period, we obtain the following fundamental units $\gamma_0 \in \mathcal O_{z,w}^1$ and $\varepsilon_0=\rho_{\alpha}(\gamma_0) \in U_{\mathcal O,z,w/F}$, which agrees with the result in Example \ref{ex1}.
\begin{align}
\gamma_0 &:=B_4B_2^{-1} \\
&= \frac{1}{2}
\begin{pmatrix}
 -(1+\eta)+ (2-\eta^2)\sqrt{\eta}&-1\\
2-\eta-\eta^2& -(1+\eta)- (2-\eta^2)\sqrt{\eta} \\
\end{pmatrix} \\
\varepsilon_0 
&= -\frac{1}{2}\left(1+\eta +\sqrt{\eta^2+2\eta-3}\right) \in U_{\mathcal O,z,w/F}.
\end{align} 
\end{enumerate}
\end{ex}

\begin{ex}\label{ex2}
\begin{enumerate}[(I)]
\item Input: $\alpha=\sqrt{\eta} +\sqrt{2\eta}$, $\beta=\sqrt{\eta} -\sqrt{2\eta}$, $\varpi:= \varpi_{\beta \to \alpha}$. 
\item[(I)$'$] The corresponding data: $z=\sqrt{\eta}$, $w=1$, $D_{z,w}=8\eta$. \\ 
The associated quadratic extension: $K_{z,w} \simeq F(\sqrt{2\eta})$. 
\item The geodesic continued fraction expansion:
\begin{align}
\varpi = \llbracket g_2; \overline{-2, 3,-3,3,-2,2,-3,3,-3,2} \rrbracket_{\Delta(2,3,7)}.
\end{align}
\item From the period, we obtain the following fundamental unit $\gamma_0$, i.e., a generator of $\mathcal O_{z,w}^1 = \Gamma_{\varpi_{\beta \to \alpha}}$ up to $\pm 1$: 
{\footnotesize
\begin{align}
\gamma_0 &:=B_{10}B_0^{-1} \\
&= 
\begin{pmatrix}
-11 - 6 \theta - 28 \theta^2 - 18 \theta^3 - 12 \theta^4 - 8 \theta^5&-8 - 22 \theta^2 - 10 \theta^4\\
-6 - 18 \theta^2 - 8 \theta^4&-11 + 6 \theta - 28 \theta^2 + 18 \theta^3 - 12 \theta^4 + 8 \theta^5 \\
\end{pmatrix}\\
&= 
\begin{pmatrix}
 -(11+28\eta+12 \eta^2)- (6+18\eta+8\eta^2)\sqrt{\eta}&-8-22\eta-10\eta^2\\
-6-18\eta-8\eta^2& -(11+28\eta+12 \eta^2)+ (6+18\eta+8\eta^2)\sqrt{\eta} \\
\end{pmatrix}.
\end{align}
}
Under the identification (\ref{ev map}), the fundamental unit $\varepsilon_0=\rho_{\alpha}(\gamma_0)$ of $U_{\mathcal O,z,w/F}$ can be written as 
\begin{align}
\varepsilon_0 
= -11-28\eta-12 \eta^2 - (6+18\eta+8\eta^2)\sqrt{2\eta} \in U_{\mathcal O,z,w/F}.
\end{align}
\end{enumerate}
\end{ex}

\begin{ex}\label{ex3}
\begin{enumerate}[(I)]
\item Input: $\alpha=2 +\sqrt{4-\eta}$, $\beta=2 -\sqrt{4-\eta}$, $\varpi:= \varpi_{\beta \to \alpha}$. 
\item[(I)$'$] The corresponding data: $z=2\sqrt{\eta}$, $w=\sqrt{\eta}$, $D_{z,w}=4\eta(4-\eta)$. \\ 
The associated quadratic extension: $K_{z,w} \simeq F(\sqrt{4\eta-\eta^2})$. 
\item The geodesic continued fraction expansion:
\begin{align}
\varpi = \llbracket g_7g_2g_7^{-1}; \overline{3,3,-2,2,-3,3,-3,3,-3,2,-2,3} \rrbracket_{\Delta(2,3,7)}.
\end{align}
\item By computing the period $\gamma_0 :=B_{10}B_0^{-1} \in \mathcal O_{z,w}^1$ we obtain the following fundamental unit $\varepsilon_0=\rho_{\alpha}(\gamma_0)$ of $U_{\mathcal O,z,w/F}$:  
\begin{align}
\varepsilon_0 
=-(28 + 80 \eta + 36 \eta^2)-(16 + 43 \eta + 19 \eta^2)\sqrt{4\eta - \eta^2}  \in U_{\mathcal O,z,w/F}.
\end{align}
\end{enumerate}
\end{ex}

\begin{ex}\label{ex4}
\begin{enumerate}[(I)]
\item Input: $\alpha=e$, $\beta=1/e$, $\varpi:= \varpi_{\beta \to \alpha}$, where $e$ is Euler's numebr.  
\item The geodesic continued fraction expansion:
\begin{align}
\varpi = \llbracket g_7^2g_2g_7^{-2};
&3,3,-3,-3,3,-3,3,-3,-3,2,-2,2,-3,3,\nonumber\\
&-2,3,2,-2,3,-3,-3,2,-2,2,-2,3,-3,\nonumber\\
&\quad \quad 2,-2,2,-2,3,2,-1,2,3,-3,-2,1,-1,\dots \rrbracket_{\Delta(2,3,7)}.
\end{align}
\end{enumerate}
The regularized $40$-th convergent $\bm x^{reg}_{40}$ is approximately $\bm x^{reg}_{40} \fallingdotseq 2.7182818284590431$, where $e$ is approximately $e \fallingdotseq 2.7182818284590452$.
\end{ex}

\subsection*{Acknowledgements}
I would like to express my deepest gratitude to Takeshi Tsuji for the constant encouragement and valuable comments during the study. 
The study in this paper was initiated during my stay at the Max Planck Institute for Mathematics in 2018. 
I would like to express my appreciation to Don Zagier for giving me the wonderful opportunity to study at the institute.
I would like to thank every staff of the Max Planck Institute for Mathematics for their warm and kind hospitality during my stay. 
This work is supported by JSPS Overseas Challenge Program for Young Researchers Grant Number 201780267 and JSPS KAKENHI Grant Number JP18J12744.

{\small
\bibliographystyle{alpha}
\addcontentsline{toc}{part}{References}

\bibliography{gcf_shimura}
}

\end{document}